\documentclass[13pt]{amsart}
\usepackage{amsmath}
\usepackage{}
\usepackage{amsfonts}
\usepackage{amsfonts,latexsym,rawfonts,amsmath,amssymb,amsthm}
\usepackage[plainpages=false]{hyperref}

\usepackage{pdfsync}

\usepackage[a4paper]{geometry}
\usepackage{esint}
\usepackage{graphics,graphicx}
\usepackage{tikz}

\numberwithin{equation}{section}

\newcommand{\beq}{\begin{equation}}
\newcommand{\eeq}{\end{equation}}
\newcommand{\beqs}{\begin{eqnarray*}}
\newcommand{\eeqs}{\end{eqnarray*}}
\newcommand{\beqn}{\begin{eqnarray}}
\newcommand{\eeqn}{\end{eqnarray}}
\newcommand{\beqa}{\begin{array}}
\newcommand{\eeqa}{\end{array}}

\def\lra{\longrightarrow}

\def\bc{\begin{center}}
\def\ec{\end{center}}

\def\begeq{\begin{equation}}
\def\endeq{\end{equation}}
\def\and{\quad{\rm and}\quad}

\let\lra=\longrightarrow

\def\mapright\#1{\,\smash{\mathop{\lra}\limits^{\#1}}\,}

\newtheorem{prop}{Proposition}[section]
\newtheorem{theo}[prop]{Theorem}
\newtheorem{lem}[prop]{Lemma}

\newtheorem{cor}[prop]{Corollary}
\newtheorem{rem}[prop]{Remark}
\newtheorem{ex}[prop]{Example}
\newtheorem{defi}[prop]{Definition}

\begin{document}
\title{Singular limits of K\"ahler-Ricci flow on Fano $G$-manifolds }

\author{Yan $\text{Li}^{*}$, Gang $\text{Tian}^{\dag}$ and Xiaohua $\text{Zhu}^{\ddag}$}

\thanks {$^*$Partially supported by China Post-Doctoral Grant BX20180010 and BIT Grant 3170012222012.}
\thanks {$^{\dag}$Partially supported by NSFC grants 11331001 and 11890661.}
\thanks {$^{\ddag}$Partially supported by NSFC Grants  11771019 and BJSF Grants Z180004.}

\address{$^*$School of Mathematics and Statistics, Beijing Institute of technology, Beijing 100081, China.}
\address{$^{\dag,\ddag}$BICMR and SMS, Peking University, Beijing 100871, China.}
\email{liyanmath@pku.edu.cn,\ \ \ tian@math.pku.edu.cn\\\ xhzhu@math.pku.edu.cn}



\begin{abstract} Let  $M$ be a Fano compactification of  semisimple complex Lie group $G$ and $\omega_0$  a $K\times K$-invariant  metric in $2\pi c_1(M)$,   where $K$ is a maximal compact subgroup of $G$. Then we     prove that the  solution of K\"ahler-Ricci flow with $\omega_0$  as an initial metric on $M$,
is of type II,   if $M$ admits no K\"ahler-Einstein metrics.
As an application, we found two Fano compactifications of $\mathrm{SO}_4(\mathbb{C})$ and one Fano compactification of $\mathrm{Sp}_4(\mathbb{C})$, on which the K\"ahler-Ricci flow will develop singularities of type II. To the authors'  knowledge, these are the first examples of Ricci flow with singularities of type II on Fano manifolds in the literature.
\end{abstract}

\maketitle

\section{Introduction}

Ricci flow was introduced by Hamilton in early 1980's and preserves the K\"ahlerian structure \cite {Ha}. The K\"ahler-Ricci flow is simply the Ricci flow restricted to K\"ahler metrics. If $M$ is a Fano manifold, that is, a compact K\"ahler manifold with positive first Chern class $c_1(M)$, we usually consider the following normalized K\"ahler-Ricci flow,
\begin{align}\label{kahler-Ricci-flow}
\frac{\partial \omega(t)}{\partial t}\, =\, -{\rm Ric}(\omega( t))\, +\,\omega(t), ~\omega( 0)=\omega_0,
\end{align}
where $\omega_0$ and $\omega (t)$ denote the K\"ahler forms of a given K\"ahler metric $g_0$ and the solutions of Ricci flow with initial metric $g_0$, respectively. \footnote{For simplicity, we will denote a K\"ahler metric by its K\"ahler form thereafter.}
It is proved in \cite{cao} that (\ref{kahler-Ricci-flow}) has a global solution $\omega(t)$ for all $t\ge 0$ whenever $\omega_0$ represents $2\pi c_1(M)$. A long-standing problem concerns the limiting behavior of $\omega(t)$ as $t\to \infty$. If $M$ admits a K\"ahler-Einstein metric $\omega_{KE}$ with K\"ahler class $2\pi c_1(M)$, then $\omega(t)$ converges to $\omega_{KE}$ (cf. \cite{tianzhu, TZ4}), but in general, $\omega(t)$ may not have a limit on $M$. A conjecture, referred as the Hamilton-Tian conjecture, was stated in \cite{T1} that any sequence of $(M, \omega(t))$ contains a subsequence converging to a length space $(M_\infty,\omega_\infty)$
in the Gromov-Hausdorff topology and $(M_\infty,\omega_\infty)$ is a smooth K\"ahler-Ricci soliton outside a closed subset $S$, called the singular set, of codimension at least $4$. Moreover, this subsequence of $(M, \omega(t))$ converges  locally to the regular part of $(M_\infty,\omega_\infty)$
in the Cheeger-Gromov topology. Recall that a K\"ahler-Ricci soliton on a complex manifold $M$ is a pair $(X, \omega)$, where $X$ is a holomorphic vector field on $M$ and $\omega$ is a K\"ahler metric on $M$, such that
\begin{align}\label{kr-soliton}{\rm Ric}(\omega)\,-\,\omega\,=\,L_X(\omega),
\end{align}
where $L_X$ is the Lie derivative along $X$. If $X=0$, the K\"ahler-Ricci soliton becomes a
K\"ahler-Einstein metric. The uniqueness theorem in \cite{TZ1, TZ5} states that a K\"ahler-Ricci
soliton on a compact complex manifold, if it exists, must be unique modulo ${\rm Aut}(M)$.\footnote {In the case of K\"ahler-Einstein metrics, this uniqueness theorem is due to Bando-Mabuchi \cite{BM}.} Furthermore, $X$ lies in the center of Lie algebra of a reductive part of ${\rm Aut}(M)$.

The Gromov-Hausdorff convergence part in the Hamilton-Tian conjecture follows from Perelman's non-collapsing result and Zhang's upper volume estimate \cite{Pe, Zh1, Zhq}. More recently, there were very significant progresses on this conjecture, first by Tian and Zhang in dimension less than $4$ \cite{TZhzh}, then by Chen-Wang \cite{Chwang} and Bamler \cite{Bam} in higher dimensions. In fact, Bamler proved a generalized version of the conjecture.

A natural problem is how regular the limit space $(M_\infty,\omega_\infty)$ is. Assuming that the Hamilton-Tian conjecture is affirmed, Tian and Zhang proved in \cite{TZhzh} that $M_\infty$ is a $Q$-Fano  variety whose singular set coincides with $S$. They proved this by establishing a parabolic version of the partial $C^0$-estimate. {\it Is this the best regularity we can have}? In fact, there was a folklore speculation that $(M_\infty, \omega_\infty)$ is actually a smooth Ricci soliton, equivalently, $S= \emptyset$. We recall that a solution $\omega(t)$ of {(\ref{kahler-Ricci-flow})} is called type I if the curvature of $\omega(t)$ is uniformly bounded, otherwise, we call $\omega(\cdot, t)$ a solution of type II. By using Perelman's entropy \cite{Pe}, we see that in the case of type I solutions, the limit $(M_\infty,\omega_\infty) $ has to be a K\"ahler-Ricci soliton. Then the above folklore speculation simply means that {(\ref{kahler-Ricci-flow})} has no type II solutions. The second-named author believed that this speculation can not be true and raised the problem of finding a Fano manifold whose K\"ahler-Ricci flow develops type II singularity at $\infty$.

In this paper, we will show that the above folklore speculation does not hold. We will prove

\begin{theo}\label{singular-type2} Let $G$ be a complex semisimple Lie group and
$M$ be a Fano $G$-manifold which admits no K\"ahler-Einstein metrics.  Then any solution of K\"ahler-Ricci flow (\ref{kahler-Ricci-flow}) on $M$ with a  $K\times K$-invariant
 initial metric $\omega_0\in 2\pi c_1(M)$ is of type II,
   where $K$ is a maximal compact subgroup of $G$.
\end{theo}

Here by a $G$-manifold, we mean a {\it (bi-equivariant) compactification of $G$} which admits a holomorphic $G\times G$-action and has an open and dense orbit isomorphic to $G$ as a $G\times G$-homogeneous space. There are examples of $G$-manifolds which admit neither K\"ahler-Einstein metrics nor K\"ahler-Ricci solitons, more precisely, we will show

\begin{theo}\label{SO(4)} There are two $\mathrm{SO}_4(\mathbb{C})$-manifolds and one $\mathrm{Sp}_4(\mathbb{C})$-manifold on which the K\"ahler-Ricci flow (\ref{kahler-Ricci-flow}) develops singularities of type II.
\end{theo}

Since $\mathrm{SO}_4(\mathbb{C})$ and $\mathrm{Sp}_4(\mathbb{C})$ are both semisimple, \footnote{In fact, $\mathrm{Sp}_4(\mathbb{C})$ is simple and $\mathrm{SO}_4(\mathbb{C})$ is semisimple, however $\mathrm{SO}_n(\mathbb{C})$ is simple when $n\ge 5$.}
Theorem \ref{SO(4)} is deduced directly from Theorem \ref{singular-type2}.
Theorem \ref{SO(4)} provides the first example of Fano manifolds on which the K\"ahler-Ricci flow develops singularity of type II and solved the problem raised by the second named author. \footnote{We would like to thank the referee for telling  us that Sz\'ekelyhidi and Delcroix have a related speculation on  the limit behavior of K\"ahler-Ricci flow  on those  Fano manifolds (cf. \cite[Page 79]{Del1}).}

 We note that the  K\"ahler metrics of  flow (\ref{kahler-Ricci-flow}) preserves $K\times K$-invariant if the initial metric  $\omega_0$ is  $K\times K$-invariant.  Then by the contradiction argument,  the proof of Theorem \ref{singular-type2} reduces to studying  the geometric deformation of $G$-manifolds with $K\times K$-invariant metrics $\omega_i$ in the  smooth topology  (cf. Proposition \ref{g-structure} and Proposition \ref{bi-holo}).  Although we shall assume the $K\times K$-invariant condition on those $\omega_i$ in proofs  of both of Proposition \ref{g-structure} and Proposition \ref{bi-holo},   the $K\times K$-invariant condition for $\omega_0$ in Theorem \ref{singular-type2} can be removed   by using a recent result for  the uniqueness problem of limits of K\"ahler-Ricci flow  in \cite{WangZ2, HL}.(cf. Theorem \ref{singular-type2-2}).

There is a way to remove the semi-simplicity condition on $G$ in Theorem \ref{singular-type2}
by examining all possible Fano $G$-compactifications which admit K\"ahler-Einstein metrics or K\"ahler-Ricci solitons with a underlying differential structure, since the Cheeger-Gromov limit is a K\"ahler-Ricci soliton by Perelman's result. For examples, this can be done for $\mathrm{SO}_4(\mathbb{C})$-manifolds
and $\mathrm{Sp}_4(\mathbb{C})$-manifolds in our case  based on the computation of associated polytopes in \cite{De12} and \cite{Ruzzi} ( also see Section 6). In fact, there are two ways to prove Theorem \ref{SO(4)} by using only Proposition \ref{g-structure}. The one is that
the volumes of $G$-manifolds of $\mathrm{SO}_4(\mathbb{C})$ ($\mathrm{Sp}_4(\mathbb{C})$) are different by the volume formula (cf. \cite{Del1, LZZ}) since volumes of corresponding ploytopes are different. Thus these Fano manifolds can not be related by jumping  complex structures. The other is to check that these Fano manifolds are all $K$-unstable (cf. Section 5). Then the limit in the flow can not be a K\"ahler-Einstein manifold, to see details in the end of proof of Theorem \ref{singular-type2}. Hence, the flow must develop singularities of type II.

 The proof of  of Theorem \ref{singular-type2} contains two main steps:  proofs of Proposition \ref{g-structure} and Proposition \ref{bi-holo}.  In  Proposition \ref{g-structure}, we prove that the   Cheeger-Gromov limit  $M_\infty$    of  $\omega_i$ is still a $G$-manifold.  Our idea is to  study the deformation of holomorphic vector fields induced by the group $G$  under $\omega_i$ (cf. Proposition \ref{vector-limit}) and to prove the limit vector fields induce an open $G$-orbit $\hat{\mathcal O}_\infty^0$ (cf. (\ref{open-g-orbit})). The $K\times K$-invariant condition of $\omega_i$  will play a crucial role  so that the convergence of toric vector fields  can control other  holomorphic vector fields  (cf. Lemma \ref{nondegenerate-vector}). In  the proof of Lemma \ref{nondegenerate-vector},   we use a technique of partial $C^0$-estimate from \cite{Ti90, T2} to compare the original metrics $\omega_i$ with the   induced metrics  by the Fubini-Study metric  from the Kodaira embeddings (cf. Lemma \ref{metric-equiv}). The advantage to use the Kodaira embeddings is:   there is  a  natural $G$-cation  on the limit space $M_\infty$  (cf. (\ref{action def})), which generates holomorphic vector fields  on    $\mathbb CP^N$ with a free  torus action  on $\hat{\mathcal O}_\infty^0$  as well as the diffeomorphisms between the complex submanifolds of   embedding images can be controlled (cf. (\ref{hat-omega})).

Proposition  \ref{bi-holo} is a   corollary  of   Theorem \ref{unique-complexstructure},  where we  prove a uniqueness result about  complex structures of  $G$-manifolds when  $G$ is semisimple. Theorem \ref{unique-complexstructure} is an independent result on the uniqueness of  complex structures,  even in the sense of $G$-equivariant  deformation of $G$-manifolds \cite{PP}. Our proof  reduces to proving a uniqueness result of  complex structures on a product of  toric manifolds  (cf. Lemma \ref{uniqueness of splitting-structure}).

The organization of paper is as follows. In Section 2, we review an existence result of K\"ahler-Einstein metrics on $G$-manifolds by Delcroix. In Section 3, we study the deformation of holomorphic vector fields induced by the group $G$ and prove Proposition \ref{vector-limit}.  Theorem \ref{singular-type2} is proved in Section 5, while   Proposition \ref{g-structure} and Theorem \ref{unique-complexstructure}   are proved in Section 4 and  Section 5, respectively.
 At last, in Section 6, we give all Fano compactifications of $\mathrm{SO}_4(\mathbb{C})$ and $\mathrm{Sp}_4(\mathbb{C})$.

\vskip5mm
\noindent{\bf Note.}  The preprint of paper was first posted in the summer of 2018 \cite{LTZ1}.  After that, we find the assumption  that $G$ is semisimple can be  removed   in Theorem \ref{singular-type2}.
In the appendix of paper, we will  give an analytic proof of  Proposition \ref{bi-holo} without this  assumption  and so we get the improvement of  Theorem \ref{singular-type2} (cf. Theorem \ref{singular-type2-new}).

\section {Preliminaries on $G$-manifolds}

In this paper, we always assume that $G$ is a reductive Lie group which is a complexification of compact Lie group $K$. Let $T^\mathbb C$ be an $r$-dimensional maximal complex torus of $G$ with its Lie algebra $\mathfrak t^{\mathbb C}$ and $\mathfrak M$ the group of characters of $\mathfrak t^{\mathbb C}$. Denote the roots system of $(G,T^\mathbb C)$ in $\mathfrak M$ by $\Phi$ and choose a set of positive roots by $\Phi_+$.
Then each element in $\Phi$ can be regarded as the one of $\mathfrak a^*$, where $\mathfrak a^*$ is the dual of the non-compact part $\mathfrak a$ of $\mathfrak t^{\mathbb C}$.

\subsection{Local holomorphic coordinates} In this subsection, we recall local holomorphic coordinates on $G$ used in \cite{Del1}. By the standard Cartan decomposition, we can decompose $\mathfrak g$ as
\begin{align}\label{g-decomposition}\mathfrak g=\mathfrak t^\mathbb C\oplus\left(\oplus_{\alpha\in\Phi}V_{\alpha}\right),
\end{align}
where $V_{\alpha}=\{X\in\mathfrak g|~{\rm ad}_H(X)=\alpha(H)X,~\forall~ H\in\mathfrak t^\mathbb C\}$ is the eigenspace of complex dimension $1$ with respect to the root $\alpha$. By \cite{Hel}, one can choose $X_{\alpha}\in V_{\alpha}$ such that $X_{-\alpha}=-\iota(X_{\alpha})$ and
$[X_{\alpha},X_{-\alpha}]=\alpha^{\vee},$ where $\iota$ is the Cartan involution and $\alpha^{\vee}$ is the dual of $\alpha$ by the Killing form.
Let $E_{\alpha}=X_{\alpha}-X_{-\alpha}$ and $E_{-\alpha}=J(X_{\alpha}+X_{-\alpha})$. Denote by $\mathfrak k_{\alpha},\,\mathfrak k_{-\alpha}$ the real line spanned by $E_\alpha,\,E_{-\alpha}$, respectively.
Then we have the Cartan decomposition of Lie algebra $\mathfrak k$ of $K$,
$$\mathfrak k=\mathfrak t\oplus\left(\oplus_{\alpha\in\Phi_+}\left(\mathfrak k_{\alpha}\oplus\mathfrak k_{-\alpha}\right)\right),$$
where $\mathfrak t=\mathfrak t^{\mathbb C}\cap \mathfrak k$ is the compact part of Lie algebra of $T^\mathbb C$.
Choose a real basis $\{E^0_1,...,E^0_r\}$ of $\mathfrak t$. Then $\{E^0_1,...,E^0_r\}$ together with $\{E_{\alpha},E_{-\alpha}\}_{\alpha\in\Phi_+}$ form a real basis of $\mathfrak k$, which is indexed by $\{E_1,...,E_n\}$. $\{E_1,...,E_n\}$ can  be also regarded as a complex basis of $\mathfrak g$.

For any $g\in G$, we define local coordinates $\{z_{(g)}^i\}_{i=1,...,n}$ on a neighborhood of $g$ by
$$(z_{(g)}^i)\to\exp(z_{(g)}^iE_i)g.$$
It is easy to see that $\theta^i|_g=dz_{(g)}^i|_g$, where $\theta^i$ is the dual of $E_i$, which is a right-invariant holomorphic $1$-form. Thus
$\displaystyle{\wedge_{i=1}^n\left(dz_{(g)}^i\wedge d\bar{z_{(g)}^i}\right)}|_g$ is also a right-invariant $(n,n)$-form, which defines a Haar measure $dV_G$.

For a smooth $K\times K$-invariant function $\Psi$ on $G$, we define a Weyl-invariant convex function $\psi$ on ${\mathfrak a}$ (called the associated function of $\Psi$ \cite{AL}) by
\begin{align}\label{relation-psi}\Psi( \exp(\cdot))=\psi(\cdot):~{\mathfrak a}\to\mathbb R.
\end{align}
The complex Hessian of the $K\times K$-invariant function $\Psi$ in the above local coordinates was computed by Delcroix as follows \cite[Theorem 1.2]{Del1}.

\begin{lem}\label{Hessian}
Let $\Psi$ be a $K\times K$ invariant function on $G$, and $\psi$ the associated function
on $\mathfrak{a}$. Let $\Phi_+=\{\alpha_{(1)},...,\alpha_{(\frac{n-r}{2})}\}$. Then for $x\in \mathfrak{a}_+=\{x'\in \mathfrak{a}|~ \alpha(x')>0,~\forall~\alpha\in \Phi_+\}$,
the complex Hessian matrix of $\Psi$ in the above coordinates is diagonal by blocks, and equals to
\begin{equation}\label{+21}
\mathrm{Hess}_{\mathbb{C}}(\Psi)(\exp(x)) =
\begin{pmatrix}
\frac{1}{4}\mathrm{Hess}_{\mathbb{R}}(\psi)(x)& 0 & & & 0 \\
0 & M_{\alpha_{(1)}}(x) & & & 0 \\
0 & 0 & \ddots & & \vdots \\
\vdots & \vdots & & \ddots & 0\\
0 & 0 & & & M_{\alpha_{(\frac{n-r}{2})}}(x)\\
\end{pmatrix},
\end{equation}
 where
\[M_{\alpha_{(i)}}(x) = \frac{1}{2}\langle\alpha_{(i)},\nabla \psi(x)\rangle
\begin{pmatrix}
\coth\alpha_{(i)}(x) & \sqrt{-1} \\
-\sqrt{-1} & \coth\alpha_{(i)}(x) \\
\end{pmatrix}.
\]
\end{lem}

\subsection{K\"ahler-Einstein metrics on $G$-manifolds}

Let $M$ be a $G$-manifold as a compactification of $G$.  We call $(M, L)$  a {\it polarized compactification} of $G$ if $L$ is a $G\times G$-linearized ample line bundle on $M$. In this paper, we just consider $L=K_M^{-1}$.
Since $M$ contains an $r$-dimensional toric manifold $Z$, there is an associated polytope $P$ of $Z$ induced by $(M,L)$, which is a lattice polytope in the lattice $\mathfrak M$ \cite{AB1, AB2}.
Let $P_+$ be the positive part of $P$ defined by $\Phi_+$ such that $P_+=\{y\in P|~ \langle\alpha,y\rangle >0, ~\forall~ \alpha\in \Phi_+\}$.
Here $ \langle\cdot,\cdot \rangle$ denotes the Cartan-Killing inner product on ${\mathfrak a^*}$. We call $W_\alpha=\{y\in\mathfrak a^* |~ \langle\alpha,y\rangle =0\}$ the Weyl wall associated to $\alpha\in \Phi_+$.

Let $\rho={\frac 1 2}\sum_{\alpha\in\Phi_+}\alpha$ be as a character in $\mathfrak a^*$ and $\Xi$ the relative interior of the cone generated by $\Phi_+$. We set a function on ${\mathfrak a^*}$ by
$$\pi(y)=\prod_{\alpha\in\Phi_+}\langle\alpha,y\rangle^2, ~y\in {\mathfrak a^*}. $$
Clearly, $\pi(y)$ vanishes on $ W_\alpha$ for each $\alpha\in \Phi_+$. Denote by $2P_+$ a dilation of $P_+$ by  $2$. We define the  barycentre of $2P_+$ with respect to the weighted measure $\pi(y)dy$ by
$${\rm bar}(2P_+)=\frac{\int_{2P_+}y\pi(y) \,dy}{\int_{2P_+}\pi(y) \,dy}.$$

In \cite{Del1}, Delcroix proved the following the existence result of K\"ahler-Einstein metrics on $G$-manifolds.

\begin{theo}\label{de} Let $M$ be a Fano $G$-manifold. Then $M$ admits a K\"ahler-Einstein metric if and only if
\begin{align}\label{bar-1}
{\rm bar}(2P_+)\in 4\rho+\Xi.
\end{align}
\end{theo}

By extending the argument for toric Fano manifolds in \cite{WZ}, Delcroix obtained a prior $C^0$-estimate for a class of real Monge-Amp\`ere equations on the cone $\mathfrak a_+\subset  \mathfrak a=\mathbb R^r$ to prove
Theorem \ref{de}, where $\mathfrak a_+=\{x\in \mathfrak a |~ \alpha(x) >0, ~\forall~ \alpha\in \Phi_+\}$. Another proof of Theorem \ref{de} was latterly  given in \cite{LZZ} by verifying the properness of $K$-energy $\mu (\phi)$ for $K\times K$-invariant K\"ahler potentials $\phi$ modulo the center $Z(G)$ of $G$. In fact, it was proved under (\ref{bar-1}) that there exist two positive constants $\delta, C_\delta$ such that
\begin{eqnarray*}
\mu (\phi)\geq \delta \inf_{\sigma\in Z(G) } I(\phi_\sigma)-C_\delta,
\end{eqnarray*}
where $I(\phi)=\int_M \phi(\omega^n -\omega_\phi^n)$ with a $K\times K$-invariant background K\"ahler metric $\omega\in 2\pi c_1(M)$, and $\phi_\sigma$ is an induced K\"ahler potential defined by $$\omega_{\phi_\sigma}=\sigma^*(\omega+\sqrt{-1}\partial\bar\partial \phi)= \omega+\sqrt{-1}\partial\bar\partial\phi_\sigma.$$

It was also showed in \cite{LZZ} that (\ref{bar-1}) is actually a $K$-stability condition in terms of \cite {T1} and \cite{D1} by constructing a $C^*$-action through a Weyl-invariant piece-wise rationally linear function.  In particular,
$M$ is $K$-unstable if ${\rm bar}(2P_+) \not\in \overline { 4\rho+\Xi}$. A more general construction of $C^*$-actions was also discussed in \cite{De12}.

\section{Deformation of holomorphic vector fields}

In this section, we give a description on  the deformation of holomorphic vector fields generated by $G$ on a $G$-manifold $M$. We introduce

\begin{defi}\label{limit-vector} Let $(M,g_i)$ be a sequence of Riemannian metrics on a closed  manifold $M$ which  has a limit  $(M_\infty, g_\infty)$ in sense of Cheeger-Gromov topology.  A sequence of tangent vector fields $(M, X_i, g_i)$ is  called  convergent to $(M_\infty, X_\infty, g_\infty)$ if
there are diffeomorphisms $F_i: M_\infty\to M$ such that
$$F_i^* g_i     \stackrel{C^\infty}{\longrightarrow} g_\infty, ~ (F_i^{-1})_* X_i \stackrel{C^\infty}{\longrightarrow} X_\infty, ~{\rm on}~ M_\infty.$$

\end{defi}

\begin{rem}\label{equivalent-convergence}
 Let $(M,g_i')$ be another convergent  sequence of Riemannian metrics on  $M$ with a Cheeger-Gromov limit  $(M_\infty, g_\infty')$, which satisfies
$$ F_i^* g_i'     \stackrel{C^\infty}{\longrightarrow} g_\infty', $$
 Then   the tangent vector field  $(M, X_i, g_i')$ is  also  convergent to $(M_\infty, X_\infty, g_\infty')$ on $M_\infty$.

\end{rem}

Let $\sigma_i$ be a sequence of auto-diffeomorphisms of $M$ and $(M,g_i)$ be a sequence of Riemannian metrics in  Definition \ref{limit-vector}. Then
$$(\sigma_i^{-1}\cdot F_i)^*  (\sigma_i^*g_i)     \stackrel{C^\infty}{\longrightarrow} g_\infty,$$
but,
the sequence of $((\sigma_i^{-1}\cdot F_i)^{-1})_*X_i$ may converge to a different limit of tangent vector fields on $M_\infty$.
 Thus Definition \ref{limit-vector} is not intrinsic in general.  However, there are  some cases in which  the limit of tangent vector fields  does not change after  holomorphic  transformations.

\begin{ex} Let $M$ be an $n$-dimensional toric manifold and $X=\sum_\alpha a_\alpha\frac{\partial}{\partial z^\alpha}$ be a torus vector field in the affine coordinates $(z^1,..., z^n)$. Then any torus action $\sigma$ on $M$ is given by
$$z\to z+z_0'$$
for some $z_0'$. Thus $\sigma_* X=X.$  Let $\sigma_i$ be a sequence of  torus actions.
We consider a sequence of torus invariant K\"ahler metrics $g_i$. Suppose that   $(M, X, g_i)$ is   convergent to $(M_\infty, X_\infty, g_\infty)$ in sense of Definition \ref{limit-vector}.  Then we still  get
$$(\sigma_i^{-1}\cdot F_i)^*  (\sigma_i^*g_i)     \stackrel{C^\infty}{\longrightarrow} g_\infty, ~((\sigma_i^{-1}\cdot F_i)^{-1})_*X = ( F_i^{-1})_*X \stackrel{C^\infty}{\longrightarrow} X_\infty. $$
\end{ex}

Now we assume that  $X$ is  a {right}-invariant holomorphic vector field on a $G$-manifold $M$ as an element of Lie algebra $\mathfrak g$ of $G$ with
${\rm im}(X)\in \mathfrak k$, where $ \mathfrak k$ is the Lie algebra of $K$. We choose a {$K\times K$-invariant} metric $\omega\in 2\pi c_1(M)$ as in Section 2. Then by the Hodge theorem, there is a real-valued smooth function $f$ (usually called a potential of $X$) on $M$ such that
$$L_X\omega=\sqrt{-1}\partial\bar\partial f.$$
  Set
\begin{align}\label{set-max}M_X^1=\{x\in M| ~ f(x)= \max_M f\}
\end{align}
and
\begin{align}\label{set-min} M_X^2=\{x\in M| ~ f(x)= \min_M f\}.
\end{align}

The following is our main result in this section.

\begin{prop}\label{vector-limit} Let $G$ be a reductive Lie group and $M$ a Fano $G$-manifold with an open {dense} G-orbit $\mathcal O$. Suppose that $(M,\omega_i, J)$ is a sequence of {$K\times K$-invariant} metrics in $2\pi c_1(M)$ which converges to a smooth limit $(M_\infty, \omega_\infty, J_\infty)$ in the Cheeger-Gromov topology.  Let $M_X^1, M_X^2$ be two sets as in (\ref{set-max}) and (\ref{set-min}), respectively. Then the following is true:
\begin{itemize}
\item[(a)]
There is a dense subset $\tilde M_X^1\subset M_X^1$ such that   any integral curve $\exp\{t{\rm re}(X)\}\cdot y_0$ generated by re$(X)$ from $y_0\in\mathcal O$ converges to a point in   $\tilde M_X^1$.  Similarly, $\exp\{t{\rm re}(-X)\}\cdot y_0$ converges to a point in a dense subset $\tilde M_X^2\subset M_X^2$.

\item [(b)] $(M, X, \omega_i)$ converges  to a non-trivial holomorphic vector field $X_\infty$ on $M_\infty$ such that
\begin{align}\label{set-decomposition} M_{\infty,X}^1\cup M_{\infty,X}^2 \subset M_{X_\infty}=\{x\in M_\infty|~X_\infty(x)=0\},
\end{align}
where $ M_{\infty,X}^1$ and $ M_{\infty,X}^2$ are the limits of $M_X^1$ and $M_X^2$ in  $M_\infty$  in the Gromov-Hausdorff topology, respectively, and both of them are non-empty and disjoint.
\end{itemize}

\end{prop}

\begin{proof}
Clearly, $M_X\subset(M\setminus\mathcal O)$ and it  can be decomposed into a  union of disjointed subsets
\begin{align}\label{decom M}M_X=\{x\in M|~X(x)=0\}=M_X^1\cup M_X^2\cup M_X^3,\end{align}
where
$$M_X^3=\{x\in M|~f(x)=c\text{ is a critical value  of $f$},  ~c\not=\max_M f\text{ or }\min_M f\}.$$
We need to show  that any {$\exp\{t\rm{re}(X)\}$-orbit} in $\mathcal O$ converges to a point in $M_X^1$.

For any $c$ above, define a level set of $f$ in $ M_X^3$ by
\begin{align}\label{mc}M_c=\{x\in M_X^3|~ f(x)=c\}.
\end{align}
Then   there is a finite  set of such $c$ such that  $M_c\neq \emptyset$ since each $M_c\subset M_X$ is an analytic subvariety of $M$.
Suppose that $M_{c_0}\neq \emptyset$ for some $c_0$ and $x\in M_{c_0}$
and assume that there is a
base point $y_0\in \mathcal O$ such that
$$\lim_{t\to\infty}  \exp\{t{\rm re}(X)\}\cdot y_0=x.$$
 Since
$f$ is increasing along the integral curve $ \exp\{t{\rm re}(X)\}  \cdot y_0$ by the relation
\begin{align}\label{mono}
\frac{df}{dt} ({\rm exp}\{t {\rm re (X)}\}\cdot y_0)=|\nabla f|^2({\rm exp}\{t {\rm re}(X)\}\cdot y_0)>0,
\end{align}
 the limit
\begin{align}\label{limit-c}c_0=\lim_{t\to\infty} f({\exp}\{t \text{re}(X)\}\cdot y_0)
\end{align}
is well defined.
On the other hand, there is an integral curve from another base point $y'\in \mathcal O$ such that
$$\lim_{t\to\infty} {\rm exp}\{t {\rm re}(X)\}\cdot y'=x'\in M_X^1$$
and
$$ f(x')=\max_M f=\lim_{t\to\infty} f({\rm exp}\{t{\rm re}( X)\}\cdot y').$$
Thus we can define a class of subsets  in $\mathcal O$ associated to numbers $c\in [c_0, \max_M f]$ by
\begin{align}\label{c-number}\mathcal O_c=\{y\in \mathcal O|~ \lim_{t\to\infty} f({\rm exp}\{t {\rm re}(X)\}\cdot y)=c\}.
\end{align}
Note that the number $c$ in (\ref{c-number}) must be a critical value of the function $f$.
 Hence, there are finitely many such numbers $c$ as in (\ref{mc}).

\textbf{Claim 1:} Each $\mathcal O_c$ is an open set.

Let $y_0=g_0\in \mathcal O_c$ and $x_0\in M\setminus\mathcal O$ be the limit of ${\rm exp}\{t {\rm re}(X)\} \cdot y_0$ as $t\to\infty$.
Then there {are} two neighborhood $U_{x_0}$ and $V_{x_0}$ with $\overline {V_{x_0}}\subset U_{x_0}$ of $x_0$ such that $|f-c|\le \epsilon$ on $U_{x_0}$ and ${\rm exp}\{t {\rm re}(X)\} \cdot y_0\subset V_{x_0}$ for any $t\ge t_0$, where $\epsilon$ is a small number and $t_0$ is a large number. Since
$${\rm exp}\{t {\rm re}(X)\}\cdot y=[{\rm exp}\{t {\rm re}(X)\}\cdot g_0]\cdot (g_0^{-1}\cdot g),~ \forall ~ y=g\in \mathcal O,$$
and $g_0^{-1}\cdot g$ is a smooth map on $M$, we see that
\begin{align}\label{p-01}{\rm exp}\{t {\rm re}(X)\}\cdot y\in U_{x_0},~\forall ~t\ge t_0, \end{align}
as long as ${\rm dist} (g_0^{-1}\cdot g,{\rm Id})<<1$.
It follows that
$$c-\epsilon\le \lim_{t\to\infty} f({\rm exp}\{t {\rm re}(X)\}\cdot y)\le c+\epsilon.$$
Thus $\lim_{t\to\infty} f({\rm exp}\{t {\rm re}(X)\}\cdot y)$ must be $c$ since there is no other critical value of $f$ near $c$.
\textbf{Claim 1} is proved.

By \textbf{Claim 1}, there are finitely many disjoint open subsets $\mathcal O_{c_j}$ such that
$$ \mathcal O= \cup_j \mathcal O_{c_j}.$$
On the other hand, from \eqref{p-01}, one can show that each {$\Omega\cap\mathcal O_{c_j}$ is a closed set for any closed set $\Omega\subset \mathcal O$}. It follows that there is only one $\mathcal O_{c_j}$ through { $\Omega\cap \mathcal O$}, and as a consequence, it must be $\mathcal O_{\max_M f}$ if {$\mathcal O_{c_j}\cap \Omega\neq \emptyset$}. Thus $c_0= \max_M f$ and $M_c$ in (\ref{mc}) with $c$ defined by  (\ref{limit-c})  must be empty. In another word, any curve $\exp\{t\text{re}(X)\}\cdot y_0$ with $y_0\in \mathcal O$ converges to a point in $M_X^1$.

Next,  we show that the set
\begin{align}\label{m-eps}
\tilde M_X^1=\{x\in M_X^1|~ x=\lim_{t\to \infty} \exp\{t\text{re}(X)\}\cdot y,~{\rm for ~ some} ~y\in\mathcal O\}
\end{align}
is dense in $M_X^1$. On contrary, if
\begin{align}\label{E-1}\mathcal E_1=\overline{\tilde M_X^1}\not=M_X^1,\end{align}
then there is another closed subset $\mathcal E_2$,  which is disjointed with $\mathcal E_1$,  such that
$$ \mathcal E_1\cup \mathcal E_2 \subset M_X^1.$$
Note
 $$M_X^1=\cap_{\epsilon>0}\mathcal M^\epsilon, $$
where
$$\mathcal M^\epsilon=\{x|\max_{\hat M_\infty}f-\epsilon< f(x)\leq\max_{\hat M_\infty}f\}.$$
Thus for sufficiently small $\epsilon$, there are two disjoint open sets $\mathcal U^\epsilon_1,\mathcal U^\epsilon_2$ such that
$$ \mathcal U^\epsilon_1\cup\mathcal U^\epsilon_2 \subset \mathcal M^\epsilon ~{\rm and} ~\mathcal E_i\subset\mathcal U^\epsilon_i ~(i=1,2)$$
with the property
$$ f\ge \max_{M_\infty}f- \epsilon, ~\forall ~~x\in {\mathcal U^\epsilon_i} ~~{\rm and}~ f \equiv \max_{M_\infty}f- \epsilon, ~\forall ~x\in \partial {\mathcal U^\epsilon_i}, ~i=1,2.$$
On the other hand, by the monotonicity \eqref{mono} and the definition \eqref{m-eps}, any integral curve generated by ${\rm re}(X)$ starting in each $\mathcal U_i$ can not leave it. Since $\mathcal U_2\cap{\mathcal O}\not=\emptyset$,   there is a point $y_0\in\mathcal U_2$ such that
$$z_0=\lim_{t\to\infty} \exp\{t{\rm re}(X)\}\cdot y_0\in\mathcal E_2\subset M_X^1,$$
which contradicts to the fact $z_0 \in\mathcal E_1$ by  (\ref{m-eps}) and \eqref{E-1}! Hence, $\tilde M_X^1$ is a dense set of $M_X^1$.
Similarly, we can show that
$$\tilde M_X^2=\{x\in M_X^2|~ x=\lim_{t\to \infty} \exp(t\text{re}(-X))\cdot y,~{\rm for ~ some} ~y\in\mathcal O\}
$$
is a dense set of $M_X^2$. {Part (a)} of the proposition is proved.

To prove Part (b) in  the proposition, we write $\omega_i$ as
$\omega_i=\omega+\sqrt{-1}\partial\bar\partial \varphi_i$ for some
$\varphi_i$. Then
$$i_X(\omega_i)=\sqrt{-1}\partial\bar\partial f_i,$$
where $f_i=f +X(\varphi_i)$.
It follows that
\begin{align}\label{x-relation}X= g^{k\bar l}(\omega_i) (f +X(\varphi_i))_{\bar l}\frac{\partial}{\partial z_k}.
\end{align}
Thus
$$ M_X=\{x\in M|~\nabla f(x)=0\}= \{x\in M|~\nabla f_i(x)=0\}. $$
Hence, $\max_M f_i= \max_M f$ and $\min_M f_i= \min_M f$ by (\ref{x-relation}). In particular,
\begin{align}\label{osc}{\rm osc}_M f_i= {\rm osc}_M f.
\end{align}

Let $h_i$ be the Ricci potential of $\omega_i$ which is normalized by
\begin{align}\label{normalization-h} \int_M e^{h_i}\omega_i^n=(2\pi c_1(M))^n.
\end{align}
Note that $\omega_i$ is $K\times K$-invariant. Then    by adding a constant $f_i$ satisfies the following  equation (cf. \cite{Fu, TZ5}),
\begin{align}\label{fi-equation}\Delta_i f_i +f_i +\langle\partial f_i, \partial h_i\rangle=\Delta_i f_i +f_i +\frac{1}{2}\langle\nabla f_i, \nabla h_i\rangle=0,
\end{align}
where $\bigtriangleup_i$ is the Laplace operator associated to $\omega_i$.
Since $h_i$ satisfies equation $\Delta_i h_i=R_i -n$, where $R_i$ is the scalar curvature of $\omega_i$,
$h_i$ is {$C^k$-uniformly} bounded for any $k$.
Thus $f_i$ is {$C^k$-uniformly} bounded associated to the metric $\omega_i$ for any $k$, and so $f_i$ converges subsequently to a smooth function $f_\infty$
on $M_\infty$. We need to show that
\begin{align}\label{osc-2}{\rm osc}_{M_\infty} f_\infty= {\rm osc}_M f.
\end{align}

\textbf{Case 1)}: There is a uniform constant $\delta_0$ such that
$$ {\rm dist}_{\omega_i}(M_X^1, M_X^2)\ge \delta_0,~\forall ~i.$$
Then there are Gromov-Hausdorff limits $ M_{\infty,X}^1$ and $ M_{\infty,X}^2$ of $M_X^1$ and $M_X^2$  in  $M_\infty$,  respectively,
such that $M_{\infty,X}^1\cap M_{\infty,X}^2=\emptyset$. Thus there are neighborhoods $B_1$ and $B_2$ of $M_{\infty,X}^1$ and $ M_{\infty,X}^2$ in $ M_{\infty}$, respectively, such that $B_1\cap B_2=\emptyset$. By  the convergence of $f_i$,  we have
\begin{align}\label{max-constnat}\max_{M_\infty} f_\infty=\sup_{B_1} f_\infty= \max_M f=f_\infty(x),~\forall~x\in M_{\infty,X}^1,
\end{align}
and
\begin{align}\label{min-constnat} \min_{M_\infty} f_\infty=\inf_{B_2} f_\infty= \min_M f =f_\infty(x),~\forall~x\in M_{\infty,X}^2.
\end{align}
Hence, we get (\ref{osc-2}).

\textbf{Case 2)}:
$$ {\rm dist}_{\omega_i}(M_X^1, M_X^2)\to 0,~{\rm as}~i\to +\infty.$$
This means that
$$M_{\infty,X}^1\cap  M_{\infty,X}^2\neq \emptyset.$$
Thus there is a point $p\in M_{\infty,X}^1\cap  M_{\infty,X}^2$ and two sequences of points $\{p_i\}$ and  $\{q_i\}$ in $M_X^1$ and  $M_X^2$, respectively, such that
$$p_i \to p~{\rm and}~  q_i\to p.$$
As a consequence, there is an open set $U_p$ around $p$ in $M_\infty$  and   two sequences  of open sets $U_{p_i}$ and  $U_{q_i}$ around $p_i$ and $q_i$ in $M$,  respectively, such that
$$U_{p_i}\to U_p~{\rm and}~U_{q_i}\to U_p$$
in the  topology of Gromov-Hausdorff.
By  the convergence of $f_i$,  we get
$$ f_\infty(p)\equiv\min_M f~{\rm and}~f_\infty(p)\equiv\max_M f.$$
This is impossible!
 Hence,  Case 2) is impossible and we prove (\ref{osc-2}) by Case 1).

By (\ref{osc-2}), we see that $\nabla f_\infty\neq 0$.
Moreover, by (\ref{fi-equation}), $f_\infty$ satisfies
\begin{align}\label{f-infty}
\Delta_\infty f_\infty +f_\infty +\langle\partial f_\infty, \partial h_\infty\rangle=0,
\end{align}
where $\Delta_\infty$ is the Laplace operator associated to $\omega_\infty$. Thus $g^{k\bar l}(\omega_\infty) (f_\infty)_{\bar l}\frac{\partial}{\partial z_k}$
defines a holomorphic vector field $X_\infty$ on $M_\infty$. On the other hand, by  the convergence of $\omega_i$, there are diffeomorphisms $F_i: M_\infty\to M$ such that
$$F_i^* \omega_i     \stackrel{C^\infty}{\longrightarrow} \omega_\infty,  ~{\rm on}~ M_\infty.$$
Since $\nabla f_i$ is $C^k$-uniformly bounded associated to $\omega_i$, we also get
 \begin{align}\label{convergence-holo-vfs}(F_i^{-1})_* (g^{k\bar l}(\omega_i) (f_i)_{\bar l}\frac{\partial}{\partial z_k})\stackrel{C^\infty}{\longrightarrow} g^{k\bar l}(\omega_\infty) (f_\infty)_{\bar l}\frac{\partial}{\partial z_k}, ~{\rm on}~ M_\infty,
 \end{align}
which is the limit of $(M, X, \omega_i)$ in the sense of Definition \ref{limit-vector}.

 By Case 1),  it is clear that
$ M_{\infty,X}^1\cap M_{\infty,X}^2=\emptyset $. Furthermore,
$$\nabla f_\infty\equiv 0, ~ \forall x\in M_{\infty,X} = M_{\infty, X}^1\cup M_{\infty,X}^2. $$
As a consequence,
$$ M_{\infty,X}\subset M_{X_\infty}.$$
Part (b) of Proposition \ref{vector-limit} is proved.

\end{proof}

\section{Deformation of  $G$-structures}

In this section, we prove  the following proposition on preservation of $G$-structures on limits in the Cheeger-Gromov topology.

\begin{prop}\label{g-structure} Let $(M_\infty, \omega_\infty, J_\infty)$ be a smooth limit of   $K\times K$-invariant  metrics $(M, \omega_i, J)$ in
$2\pi c_1(M)$ on a Fano $G$-manifold $M$ in the Cheeger-Gromov topology as in Proposition \ref{vector-limit}.
Then $(M_\infty, \omega_\infty)$ is also a Fano $G$-manifold.
\end{prop}

We use the Kodaira embedding to prove the proposition. This means that there exists an integer $m$ such that $M$ can be embedded into $\mathbb CP^N$ by a unitary orthogonal   basis  $\{s_A^i, ~i=1,..., N+1\}$ of holomorphic sections of $K_M^{-m}$  with  $L^2$-norm  induced by
$\omega_i$. Note that $M_\infty$ is diffeomorphic to $M$. Thus by the convergence of $\omega_i$, we see that
there is a uniform integer $N$ independent of $i$ such that the following properties are satisfied:

1) There is a holomorphic embedding $\Phi_i$ from $M$ to $\mathbb CP^N$ for each $(M,\omega_i)$.

2) There is a holomorphic embedding $\Phi_\infty$ from $M_\infty$ to $\mathbb CP^N$ for the K\"ahler manifolds $(M_\infty, \omega_\infty)$.

3) The image $\Phi_i(M)=\hat M_i$ converges to the image $\Phi_\infty(M_\infty)=\hat M_\infty$ according to the topology of complex submanifolds.

The above properties 1)-3) come from the partial $C^0$-estimate as in \cite{DS, Ti13a, T2}  (or simply a variant of Tian's almost isometry theorem \cite{Tian89J}) for a sequence of K\"ahler-Einstein metrics or conical  K\"ahler-Einstein metrics.  We note that the curvature of $\omega_i$ in Proposition \ref{g-structure} is uniformly bounded and  the sequence $\{\omega_i\}$ can be regarded as  a special case in their papers. As a consequence, the norm of sections $s_A^i$ with respect to $\omega_i$ as  functions on $M_\infty$ is uniformly $C^\infty$-bounded and so the basis $\{s_A^i\}$ is $C^\infty$-convergent to a basis of $\{s_A^\infty\}$ on  $H^0(M_\infty, K_{M_\infty}^{-m})$.  Thus by 3),  we can choose   a covering $\{U_\alpha\}$  of $ \hat M_\infty$  with local  holomorphic coordinates and   diffeomorphisms   $F_{i}: \hat M_\infty\to \hat M_i$ such that for each $\hat M_i$ there is a  covering  $\{U_\alpha^i\subset F_{i}(U_\alpha)\}$   with local  holomorphic coordinates and uniform norms of transformation functions. Moreover,
\begin{align}\label{hat-omega} F_i^* \hat\omega_i     \stackrel{C^\infty}{\longrightarrow} \hat\omega_\infty,
\end{align}
 where   $\hat\omega_i\,=\, \frac{1}{m} \omega_{FS}|_{ \hat M_i }$ and  $\hat\omega_\infty\,=\, \frac{1}{m} \omega_{FS}|_{ \hat M_\infty }$.

 In the following, we  compare $\omega_i$ with the induced  metric
 $\hat \omega_i$.
 Write
$$(\Phi_i^{-1})^*\omega_i = \hat\omega_i+\sqrt{-1} \partial\bar\partial\varphi_i$$
 for some K\"ahler potential $\varphi_i$ in  $\hat M_i$.
Then by using  the  regularity theory  of   complex Monge-Amp\`ere equation, we prove

\begin{lem}\label{metric-equiv}
There is a uniform constant $A>0$ such that
\begin{align}\label{induced-metrics}A^{-1} (\Phi_i^{-1})^*\omega_i \le  \hat \omega_i\le A  (\Phi_i^{-1})^*\omega_i,~{\rm on}~\hat M_i.
\end{align}
Moreover,  for any integer $k>0$ it holds
\begin{align}\label{higher-regularity}
\|\varphi_i \|_{C^{k}(\hat M_i)}\le     A_k,
\end{align}
where $A_k$ is a uniform constant independent of $i$.

\end{lem}

\begin{proof}
Let  $h_i$ and $\hat  h_i$ be the Ricci potentials  of  $\omega_i$ and $\tilde\omega_{i}$, respectively.   Then by the  convergence of $\omega_i$ and $\hat \omega_{i}$,  both of $h_i$ and $\tilde h_i$ are uniformly bounded.
   Moreover,  $\varphi_i$ (maybe  different by a constant)  satisfies
the following  complex Monge-Amp\`ere equation,
\begin{align}\label{potenial-equation}
(\hat \omega_{i}+\sqrt{-1}\partial\bar\partial \varphi_i)^n=
e^{-\varphi_i+\hat h_i-h_i}\hat \omega_{i}^n, ~{\rm in}~ \hat M_i.
\end{align}
By   the partial $C^0$-estimate and gradient estimate of $\{s_\alpha^i\}$ (cf. \cite{T2}), we know  that
$$  |\varphi_i|\le C~{\rm and}~\hat \omega_i\le A  (\Phi_i^{-1})^*\omega_i.$$
Thus  by  (\ref{potenial-equation}) we also get
$$A^{-1} (\Phi_i^{-1})^*\omega_i \le \hat \omega_i$$
possibly by choosing a  bigger $A$.  Hence, (\ref{induced-metrics}) is true.

Note that
 $$\Delta_{\omega_{i}}h_i=R_{i}-n, ~{\rm in}~ M, $$
 where $R_i$ is the scalar  curvature of $\omega_i$, which is uniformly bounded.
 By  (\ref{induced-metrics}), we have
 $$|\Delta_{\hat \omega_{i}} h_i|\leq C,~{\rm in}~\hat M_i.$$
 It follows that
 \begin{align}
\| h_i\|_{C^{1,\alpha}(\hat M_i)} \le C_1. \notag
\end{align}
Hence, the regularity theory of (\ref{potenial-equation})  (cf. \cite{Wang})  implies that
$$\|\varphi_i \|_{C^{3,\alpha}(\hat M_i)}\le C_3.$$
Repeating the above argument,  we will  get (\ref{higher-regularity}).

\end{proof}

By  (\ref{hat-omega}) together with Lemma \ref{metric-equiv}, we  get
\begin{align}\label{f-maps} F_i^* ((\Phi_i^{-1})^*\omega_i)     \stackrel{C^\infty}{\longrightarrow} (\Phi_\infty^{-1})^*\omega_\infty.
\end{align}
Let $\{E_1,...,E_n\}$ be a basis of the Lie algebra $\mathfrak g$.  Then  the left (right)  action of $G$ induces a  space  ${\rm span}\{ e_1,..., e_n\}$ of holomorphic vector fields with ${\rm im}(e_a)\in\mathfrak k$ on $ M$,  and so the holomorphism $\Phi_i$ induces  a  space  ${\rm span}\{\hat e_1^i,...,\hat e_n^i\}$ of holomorphic vector fields on $\hat M_i$.
Since by the Part (b) of Proposition \ref{vector-limit}, for each $a$, $( e_a, \omega_i)$ converges to a holomorphic vector field $( e_{a}^{\infty}, \omega_\infty)$ on $M_\infty$, $(\hat e_a^i, (\Phi_i^{-1})^*\omega_i)$ converges to a holomorphic vector field $( \hat e_{a}^{\infty}, (\Phi_\infty^{-1})^* \omega_\infty)$ on $ \hat M_\infty$.  In fact, by (\ref{f-maps}),
\begin{align}\label{induced-vfs} (F_i^{-1})_*  \hat e_{a}    \stackrel{C^\infty}{\longrightarrow} \hat e_{a}^{\infty}= (\Phi_\infty)_*e_{a}^{\infty},~\forall ~a=1, ..., n.
\end{align}
Hence, by Remark \ref{equivalent-convergence}, it follows that
$( \hat e_a^i,  \hat\omega_i)$ converges to a holomorphic vector field $( \hat e_{\alpha}^{\infty}, \hat \omega_\infty)$ on $\hat M_\infty$.
As a consequence, for each $\hat e_a^i$, the holomorphic
 coefficients of  $\hat e_a^i$ are uniformly bounded under
local holomorphic coordinates on $U_\alpha^i$.

 For any $g\in G$, the map $\Phi_i$ induces a left (right) action on
 $$H^0(M, K_M^{-m})={\rm span}\{s_A^i,~i=1,..., N+1\},$$
thus $G$ can be regarded as a subgroup of ${\rm PGL}(N+1, \mathbb C)$ induced by the map $\Phi_i$.
Those subgroups induced by different $\phi_i$ are conjugate to each other by automorphisms induced by $\sigma_{ij}^{-1}\cdot G\cdot\sigma_{ij}$,
where $\sigma_{ij} \in  {\rm PGL}(N+1, \mathbb C)$ induced by the Kodaira embeddings $\phi_i$ and $\phi_j$.  Without confusion, we still denote by $G$ each of such subgroups which may vary on $\phi_i$.
Furthermore, any one-parameter subgroup $\sigma_t$ generated by ${\rm Im}( e_a^i)$ induces a family of isomorphisms on $H^0(M, K_M^{-m})$. Taking the derivative on $t$,  we get a lifting holomorphic vector field of $\hat e_a^i$   on  $\mathbb CP^N$.

For any $\hat x_\infty\in \hat M_\infty$ and any sequence $\{\hat x_i\}$ such that
$$\hat x_i\in\hat M_i\text{ and }\hat x_i\to\hat x_\infty,$$
we  define
\begin{eqnarray}\label{action def}
g(\hat x_\infty)=\lim_{i\to\infty} g(\hat x_i)\in \hat M_\infty,~\forall ~g\in G.
\end{eqnarray}
Using the convergence of holomorphic vector fields $( \hat e_a^i,  \hat\omega_i)$ on $M$, one can easily show that the limit $g(\hat x_\infty)$ is independent of the choice of $\{\hat x_i\}$.  Moreover, we have
\begin{eqnarray}\label{communictae}
g'(g(\hat x_\infty))= (g'\cdot g)(\hat x_\infty),~\forall ~g, g'\in G.
\end{eqnarray}
Thus,  by (\ref{action def}) we  define a left $G$-action on $\hat M_\infty$, which induces the one
on $ M_\infty$ by the relation $\Phi_\infty \cdot g= g\cdot\Phi_\infty$ through the holomorphism $\Phi_\infty$.
Similarly, we can define a right $G$-action on $M_\infty$ and so get a $G\times G$-action
on $ M_\infty$.

The following lemma shows that $G$  acts on $\hat M_\infty$ effectively. Namely,  $\{\hat e_{1}^{\infty},..., \hat e_{n}^{\infty} \}$ becomes a basis of Lie algebra of $G$ acting on $\hat M_\infty$.

\begin{lem}\label{vector-independent} $\hat e_{1}^{\infty}, ..., \hat e_{n}^{\infty}$ are all linearly  independent on $\hat M_\infty$.
\end{lem}

\begin{proof} It suffices to prove that $e_{1}^{\infty}, ...,  e_{n}^{\infty}$ are all  linearly independent $ M_\infty$.
In fact, if
\begin{align}\label{independent} \sum_\alpha a_\alpha  e_{\alpha}^{\infty}\equiv 0,~{\rm for~some}~a_\alpha\neq 0,
\end{align}
then,  by the Part (b) of Proposition \ref{vector-limit}, the vector field $\sum_\alpha a_\alpha  e_{\alpha}^i$
converges to a nontrivial holomorphic vector field, which should be $ \sum_\alpha a_\alpha  e_{\alpha}^{\infty}$ on $ M_\infty$. This is a contradiction with (\ref{independent}).

\end{proof}

Let $\mathcal O$ be an open dense $ G$-orbit in $M$. Since $M$ has finitely many $G\times G$-orbits \cite{AB1,AB2},
there are basis points $x_\delta\in M\backslash\mathcal O$, $\delta=1,...,k$, such that
\begin{align}\label{limit equivariant} M= \mathcal O \cup_{\delta} (G\times G)x_\delta.
\end{align}
Note that the closure of each { $G\times G$-orbit $(G\times G)x_\delta$ } is a smooth algebraic variety whose dimension is less than $n$. Then up to a subsequence, the closure of $\Phi_i((G\times G)x_\delta)$ converges to an algebraic limit in $\mathbb CP^N$. As a consequence,
$\Phi_i(M\backslash\mathcal O)$ has an algebraic limit $D \hat M_\infty$ in  $\hat M_\infty\subset \mathbb CP^N$.

For any $i$ and $g\in G\times G$,  we have
\begin{eqnarray}\label{equivariant}
g\circ\Phi_i=\Phi_i\circ g.
\end{eqnarray}
Then by (\ref{action def}) and (\ref{limit equivariant}), for any $\hat x_{\infty}\in \hat M_\infty$ there is a sequence of $g_i \in G\times G$ such that
\begin{align}\label{limit-sequence}
{\hat x_{\infty}=\lim_i g_i\cdot \Phi_i( x_0), ~{\rm or}~ \hat x_{\infty}=\lim_i g_i\cdot \Phi_i( x_\delta),~{\rm for ~some}~\delta\in \{1,...,k\},}
\end{align}
where $x_0 \in \mathcal O$. We define an open set in $\hat M_\infty$ by $\hat{\mathcal O}_\infty=\hat M_\infty\setminus D \hat M_\infty$. Thus  for any $\hat x_{0,\infty}\in \hat{\mathcal O}_\infty $, there exists a $\delta_0>0$ such that
$${\rm dist} ( \hat x_{0,\infty}, D\hat M_\infty
)\ge 2\delta_0>0.$$
It follows that there are $\hat x_{0,i} \in \hat M_i$ such that
${\rm dist} ( \hat x_{0,i}, \Phi_i(M\backslash\mathcal O))\ge \delta_0>0.$

Since $G$ induces an action on $H^0(M,K_M^{-m})$ for each $i$ through an orthonormal basis corresponding to $\phi_i$, by taking limit as in
(\ref{action def}), $G$ induces an action on $H^0(M_\infty, K_{M_\infty}^{-m})={\rm span}\{s_1^\infty,..., s_{N+1}^\infty\}$. It follows that each holomorphic vector field $\hat e_{a}^{\infty}$, where $a=1,\cdots,n$, can be lifted to a vector field on $\mathbb CP^N$ as $\hat e_{a}^{i}$ does.  Namely, $\{ \hat e_1^\infty, .., \hat e_n^\infty\} $  can be induced by a basis of the Lie algebra of $G$ which acts on  $\hat M_\infty$.
In particular, a maximal $r$-dimensional torus subgroup $T^\mathbb C$  of $G$ acting on $\hat M_\infty$, which generated by  a basis $\{X_1,...,X_r\}$ of {$\mathfrak a$},  induces an $r$-dimensional torus subgroup  $\tilde T^\mathbb C$ of  ${\rm PSL}(N+1, \mathbb C)$ generated by an $r$-dimensional torus vector fields on  $\mathbb CP^N$.

Let $\tilde W_1,...,\tilde W_{N+1}$ be the $(N+1)$ hyperplanes in $\mathbb CP^N$ where $\tilde T^\mathbb C$ does not acts freely.
Then  for any induced holomorphic vector field $\tilde X$ of $X$ in $\mathfrak t^{\mathbb C}$ on $\hat M_\infty$, it holds
\begin{align}\label{zero-set-x}\{\hat x\in \hat M_\infty|~ \tilde X(\hat x)=0\}\subset \cup_{\alpha}\tilde W_\alpha.
\end{align}
Set
\begin{align}\label{open-g-orbit}\hat{\mathcal O}_\infty^0=\hat{\mathcal O}_\infty\backslash(\cup_\alpha\tilde W_\alpha).
\end{align}
To prove that $(M_\infty,J_\infty)$ is a $G$-manifold, it suffices to show that $\hat{\mathcal O}_\infty^0$ is isomorphic to $G$. Without loss of generality, we may assume that $ {\hat x}_{0,\infty}\in \hat{\mathcal O}_\infty^0$ above.  Thus  it reduces to proving  that $G$ acts on $\hat x_{0,\infty}$ freely and $G\cdot \hat x_{0,\infty}=\hat{\mathcal O}_\infty^0$.

The following key lemma shows that any holomorphic vector field induced by $G$ is non-degenerate on $\hat{\mathcal O}_\infty^0$.

\begin{lem}\label{nondegenerate-vector} For any induced holomorphic vector field $\tilde X$ of $X\in \mathfrak g$ on $\hat M_\infty$, it holds
\begin{align}\label{non-zero} \tilde X(\hat x_\infty)\neq 0, ~ \forall ~\hat x_\infty\in \hat{\mathcal O}_\infty^0.
\end{align}
\end{lem}

\begin{proof}

Let $\{E_1,...,E_n\}$ be a basis of $\mathfrak g$ such that $X_1=JE_1,..., X_r=JE_r\in\mathfrak a$, and $E_{a'}=E_\alpha$ and $E_{a'+ \frac{n-r}{2} }=E_{-\alpha}$, $a'=r+1,..., \frac{n+r}{2}$, which satisfy $V_\alpha\oplus V_{-\alpha}=\text{span}\{E_\alpha,E_{-\alpha}\} $ as in (\ref{g-decomposition}),
where $V_\alpha$ are eigenvectors associated to the positive roots $\alpha\in\Phi_+$. Let  $\{ e_1,..., e_n\}$ be  holomorphic vector fields with ${\rm im}(e_a)\in\mathfrak k$ on $ M$ induced by $\{E_1,...,E_n\}$.
Then the induced holomorphic vector fields $\{\hat e_1^i,...,\hat e_n^i\}$ on $\hat M_i$ by $\Phi_i$ converge to a basis
$\{\hat e_{1}^{\infty},..., \hat e_{n}^{\infty} \}$ of holomorphic vector fields  on $\hat M_\infty$
as in (\ref{induced-vfs}).

Let
$$\mathcal O_\infty^0=\Phi_\infty^{-1}(\hat {\mathcal O}_\infty^0)\subset M_\infty.$$
Note that
$$e_{a}^{\infty}= (\Phi_\infty^{-1})_*\hat e_a^{\infty}, ~a=1,...,n.$$
Then (\ref{non-zero}) is equivalent to
 \begin{align}\label{non-zero-2}\bar X(x_\infty)\neq 0, ~ \forall ~ x_\infty\in {\mathcal O}_\infty^0,
 \end{align}
 where $\bar X\in {\rm span}\{e_{1}^{\infty}, ...,e_{n}^\infty \}$ is any nontrivial vector field.
  By (\ref{zero-set-x}),   we have already known  that    (\ref{non-zero-2})  holds   for any non-trivial $\bar X\in {\rm span}\{e_{1}^{\infty}, ...,e_{r}^\infty \}$.
More precisely, we have
\begin{align}\label{non-degenerate-vf}
\omega_{\infty}( \bar X, \bar X)(x_\infty)&=(\Phi_\infty^{-1})^*\omega_{\infty}( \tilde X, \tilde X)(\hat x_\infty)\notag\\
&\ge c\hat\omega_{\infty}( \tilde X, \tilde X) (\hat x_\infty)\neq 0,
\end{align}
where
$\tilde X=(\Phi_\infty)_*\bar X\in {\rm span}\{\hat e_{1}^{\infty}, ..., \hat e_{r}^\infty \}.$
In the following, we want to show that (\ref{non-degenerate-vf}) implies  (\ref{non-zero-2}).

 By  (\ref{induced-vfs}) and (\ref{non-degenerate-vf}),  there are  a constant  $c'>o$ and  a small ball $B_\delta(\hat x_{\infty})\subset \mathbb CP^{N}$ near $ \hat x_{\infty}$
 such that
\begin{align} (\hat\omega_i(  \hat e_{a}^i, \hat e_{b}^i))_{r\times r}\ge\frac{1}{2}  (\hat\omega_{\infty}( \hat e_{a}^\infty,
 \hat e_{b}^\infty))_{r\times r}\ge c' {\rm Id}, ~{\rm in}~ B_\delta(\hat x_{\infty})\cap \hat M_i.\notag
\end{align}
Write
$$ \omega_i=\sqrt{-1}\partial\bar\partial\psi^i=\sqrt{-1}\partial\bar\partial\Psi^i $$
for some  convex function $\psi^i$ on ${\mathfrak a}$ as in
(\ref{relation-psi}).
Then by  Lemma \ref{metric-equiv},  we get
\begin{align}\label{converge-metric}
(\psi_{ab}^i)& =
(\omega_i(e_a, e_b))\notag \\
&=  ((\Phi_i^{-1})^*\omega_i( \hat e_{a}^i, \hat e_{b}^i))\notag\\
&\ge A^{-1}(\hat\omega_i(  \hat e_{a}^i, \hat e_{b}^i)) \ge\delta_0 {\rm Id}, ~{\rm in}~ \Phi_i^{-1}( B_\delta(\hat x_{\infty})\cap \hat M_i),
\end{align}
where $\delta_0>0$ is   a uniform constant.
Note that $|e_a|_{\omega_i}$ is uniformly bounded as in the proof of (\ref{convergence-holo-vfs}).
Hence,  we derive
\begin{align}
\delta_0 {\rm Id}\le ( \psi^i_{ab}) \le \frac{1}{\delta_0} {\rm Id}, ~{\rm in}~ \Phi_i^{-1}( B_\delta(\hat x_{\infty})\cap \hat M_i), \notag
\end{align}
as long as $\delta_0$ is small enough. Choose $\hat x_i\in \hat M_i\to \hat x_\infty=\Phi_\infty(x_\infty)\in\hat M_\infty$ and let  $x_i=\Phi_i^{-1}(\hat x_i)\in M$. Then $x_i\to x_\infty\in M_\infty$ in the  Gromov-Hausdroff topology. Note that  $B_\delta(\hat x_{\infty})\cap \hat M_i$ contains a uniform small  geodesic ball  centered at $\hat x_i$ associated to the metric $\hat\omega_i$.
Therefore,  again by   Lemma \ref{metric-equiv}, it is easy to  see that there is  a sequence of open sets $U_{x_i}\subset \mathcal O$, each of which  contains an $\epsilon$-geodesic ball  $(B_{\epsilon}(x_i),  \omega_i)$  centered at $x_i$ associated to $\omega_i$,  where the radius  $\epsilon$ is a uniform   small constant,   such that
\begin{align}\label{a0}
\delta_0 {\rm Id}\le (\psi^i_{ab}) \le \frac{1}{\delta_0} {\rm Id}, ~{\rm in}~ U_{x_i}.
\end{align}

\textbf{Claim 1:}  There is a   uniform   small constant $\epsilon_0$ such that $\Delta_{2\epsilon_0}\subset (B_{\epsilon}(x_i),\omega_i)\subset U_{x_i}$,
where $\Delta_{2\epsilon_0}=\{z=(z^1,..., z^n)|~|z^l-x_i^l|< 2\epsilon_0\}$ is an  $2\epsilon_0$-square of dimension $n$  centered at $x_i$ in the  local coordinates $\{z_{(g)}^l\}_{l=1,...,n}$ on  $\mathcal O$ introduced  in  Section 2.1.

In fact, as in the proof of (\ref{convergence-holo-vfs}), we see that each  of
\begin{align}\label{e-norm}|e_a|_{\omega_i}^2= \omega_i(\frac{\partial}{
\partial z^a},\overline{\frac{\partial}{
 \partial z^a}}) ~(a=1,...,n)
 \end{align}
 is uniformly bounded. In particular,  $\omega_i$  as a metric tensor is uniformly bounded under the  local coordinates $\{z_{(g)}^l\}_{l=1,...,n}$ in  $(B_{\epsilon}(x_i),  \omega_i)$. Then   \textbf{Claim 1} follows immediately.

\textbf{Claim 2:}   There is a uniform constant $c_0>0$ such that
\begin{align}\label{psi-ab} c_0\le  \langle\alpha, \nabla\psi^i(x)\rangle \coth\alpha(x)\le \frac{1}{c_0}, ~\forall ~x\in \Delta_{\epsilon_0}.
\end{align}

 By (\ref {+21}) and (\ref{e-norm}), we have
 \begin{align}\label{norm-e}|e_\alpha|_{\omega_i}^2= \langle \alpha, \nabla\psi^i(x)\rangle \coth\alpha(x).
 \end{align}
 Then  the upper bound of  (\ref{psi-ab}) is true. Thus it suffices to get the lower bound.

 Case 1:
$\alpha(x)<<1$.  Then  there exists  $ x'\in W_\alpha$ on a Weyl wall $W_\alpha$  such that
$x=x'+t\alpha$ for some small $t\ge 0$. By the fact  $ \langle \alpha, \nabla\psi^i(x')\rangle =0$,   it follows that
$$ \langle \alpha, \nabla\psi^i(x)\rangle  = \alpha^a\alpha^b \psi^i_{ab}((x'+t'\alpha))|\alpha|^2t,$$
where $t'\le t$.   Since $\psi^i_{ab}((x'+t'\alpha))$ satisfies (\ref{a0})  by \textbf{Claim 1} and (\ref{a0}), we get
$$ \langle \alpha, \nabla\psi^i(x)\rangle \coth\alpha(x)= \alpha^a\alpha^b \psi^i_{ab}((x'+t'\alpha)) |\alpha|^2 t \coth(|\alpha|^2t)\ge  c_0.$$

 Case 2: there exists a $\delta_{0}'>0$ such that
$\alpha(x)\ge \delta_0'$ for any $x\in \Delta_{2\epsilon_0}$ ($\epsilon_0$ may be replaced by a smaller number if necessary). We need to prove
\begin{align}\label{lower-bound-gradient}
 \langle \alpha, \nabla\psi^i(x)\rangle\ge c_0', ~\forall~ x\in \Delta_{\epsilon_0},\end{align}
where $c_0'>0$ is a small uniform constant.
On contrary, there is a sequence of $y_i\in \Delta_{\epsilon_0}$ such that
$$ \langle \alpha, \nabla\psi^i (y_i) \rangle  \to 0.$$
Then we choose $z_i=y_i-t_0\alpha\in \Delta_{2\epsilon_0}$ such that
$$\alpha(z_i)\ge \frac{\delta_0'}{2}.$$
Thus $z_i\in \mathfrak a_+$. On the other hand, we have
$$  \langle \alpha, \nabla\psi^i(z_i)\rangle = \langle \alpha, \nabla\psi^i (y_i)\rangle  -\alpha^a\alpha^b \psi^i_{ab}(y_i-t'\alpha)|\alpha|^2t_0,$$
where $t'\le t_0$.
By  \textbf{Claim 1} and (\ref{a0}), we get
$$ \langle \alpha, \nabla\psi^i(z_i) \rangle  <  0,~i>>1.$$
This is impossible since $z_i\in \mathfrak a_+$! Thus (\ref{lower-bound-gradient}) is true and we also get the lower bound of (\ref{psi-ab}) since
$ \coth\alpha(x)>1$. \textbf{Claim 2} is proved.

Now we can complete the proof of Lemma \ref{nondegenerate-vector}.  By Lemma \ref {Hessian}, we have
\begin{align}\label{normal-e}
&\langle e_{a'}, e_{b'} \rangle_{  \omega_i} \equiv 0,~a'\neq b', ~a',b'=r+1,...,\frac{n+r}{2},\notag\\
&\langle e_{a'},  e_a\rangle_{  \omega_i}\equiv 0,~a=1,...,r,~a'=r+1,...,n.
\end{align}
Thus it remains to  check that for each $\alpha$ it holds
\begin{align}\label{non-zero-sub}\bar X(x_\infty)\neq 0, ~ \forall ~ \bar X\in {\rm span}\{e_{\alpha}^{\infty}, e_{-\alpha}^\infty \}.
 \end{align}
By (\ref{norm-e}) and  \textbf{Claim 2}, we see that
\begin{align}\label{relation-non-vanishing} e_{\alpha}^{\infty}(x_\infty)\neq 0~{\rm and}~ e_{ -\alpha}^\infty(x_\infty) \neq 0.
\end{align}
Hence, we need to show that  $e_{\alpha}^{\infty}$ and  $e_{-\alpha}^\infty$ are independent.

Case 1: $\alpha(x_i)\to \infty$. Then
$$\alpha(x)\to \infty, ~\forall~x\in \Delta_{2\epsilon_0}.$$
Moveover, by (\ref{lower-bound-gradient}),
there exists a $c_{0}'>0$ such that
$$ \langle \alpha, \nabla\psi^i(x)\rangle   \ge c_0', ~\forall~ x\in \Delta_{\epsilon_0}.$$
Thus  the matrix block $M^i_{\alpha}(x)$ of $\omega_i$ on  $\Delta_{\epsilon_0}$ in (\ref{+21}) uniformly converges to
\[h(x)
\begin{pmatrix}
1 & \sqrt{-1} \\
-\sqrt{-1} & 1 \\
\end{pmatrix},
\]
where $h(x)$ is a smooth positive function  $\Delta_{\epsilon_0}$. Note that the above matrix is degenerate and both of $e_{\alpha}^{\infty}$ and  $e_{-\alpha}^\infty$ do not vanish in $(B_{\epsilon}(x_\infty),  \omega_\infty)$ by (\ref{relation-non-vanishing}).  Hence,  $e_{\alpha}^{\infty}$ and  $e_{-\alpha}^\infty$ must be linearly dependent at $x_\infty \in (B_{\epsilon}(x_\infty),  \omega_\infty)$. In fact, we have
 $$|e_{\alpha}^i -\sqrt{-1}e_{-\alpha}^i|(x_i)\to 0,~{\rm as}~i\to\infty, $$
and  we get
\begin{align}
e_{\alpha}^\infty(x_\infty)= \sqrt{-1}e_{-\alpha}^\infty(x_\infty).\notag
\end{align}

 Let
 $$U=\{z=(z^1,...,z^n)|~|z^a-x_i^a|\le \epsilon_0,~a=1,...,r\}\subset M$$
 be a subset in $M$ and
 $$\pi: U\to U'= \{ z'=(z^1,...,z^r)|~|z^a-x_i^a| < \epsilon_0,~a=1,...,r\}\subset\mathfrak a $$
 be   the projection. Then  for any curve $\gamma$ starting from $x_i$ such that $\pi(\gamma)\cap \partial U'\neq \emptyset$,
  it is easy to see  by  (\ref{a0}) and (\ref{normal-e}) that   there is a uniform small constant $\epsilon'$ such that
 $${\rm length}(\gamma)\ge \epsilon'.$$
 On the other hand,  for any minimal geodesic  ray $\gamma$ starting from $x_i$ such that $\pi(\gamma)\cap \partial U'= \emptyset$ its length is bigger than
 $r_{\rm inj}(x_i)$ which has a uniform lower bound by the convergence of $\omega_i$.  Thus    there is a uniform  constant $\epsilon'$ such that the $\epsilon'$-geodesic ball  $(B_{\epsilon'}(x_i),  \omega_i)$  is contained in $U$.
Since the convergence of  $M^i_{\alpha}(x)$ is independent  of the coordinate variables  of $z^{a'}$ ($a'=r+1,...,n$), we can actually  prove that  $e_{\alpha}^{\infty}$ and  $e_{-\alpha}^\infty$ are   globally linearly dependent in $(B_{\epsilon'}(x_\infty),  \omega_\infty).$
But this is impossible by Lemma \ref{vector-independent}.  In the other words, Case 1  will not happen. Therefore,  we need to consider the following case.

Case 2:  $\alpha(x_i)\le A$ for some uniform constant $A$.
Then
\begin{align}\label{orthogonal}
 & \tilde e^i_\alpha=\frac{\sqrt2 e_\alpha}{\sqrt{ \langle \alpha, \nabla\psi^i(x_i)\rangle  \coth\alpha(x_i)}},\notag\\
& \tilde e^i_{-\alpha}=\frac{\sqrt2( e_{-\alpha}-\frac{\sqrt{-1}}{\coth\alpha(x_i)} e_\alpha)}{\sqrt{  \langle \alpha, \nabla\psi^i(x_i)\rangle \left(\coth\alpha(x_i)-\coth^{-1}\alpha(x_i)\right)}}
\end{align}
form a unitary  orthogonal basis on $ {\rm span}\{ \tilde e_{\alpha}, \tilde e_{-\alpha} \} ( x_i) \subset (TM|_{ x_i}, \omega_i)$.
As in the proof of    {(\ref{psi-ab})}, there is a uniform constant $a_0>0$ such that
$$a_0\le \langle \alpha, \nabla\psi^i(x_i)\rangle  \left(\coth\alpha(x_i)-\coth^{-1}\alpha(x_i)\right)\le a_0^{-1}.$$
Thus  the potentials $\tilde f_{\alpha}^i$  ( $\tilde f_{-\alpha}^i$)  of $\tilde e^i_{\alpha}$ associated to $\omega_i$ are uniformly bounded by (\ref{orthogonal}). Hence, as in the proof of Part (b) in Proposition (\ref{vector-limit}),
$\{\tilde e^i_{\alpha},\tilde e^i_{-\alpha}\}$ converges to a subbasis $\{\tilde e^\infty_{\alpha}, \tilde e^\infty_{-\alpha}\}$,   which is orthogonal and unitary at $x_{\infty}$ with respect to $ \omega_\infty$.  Moreover,  again  by (\ref{orthogonal}), we see that
 there are constants $a\neq 0,d\neq 0, c$ such that
\begin{align}
  \tilde e^\infty_\alpha=a e^\infty_\alpha,~
\tilde e^\infty_{-\alpha}= d( e^\infty_{-\alpha}- ce^\infty_\alpha).
\end{align}
Therefore,   $e_{\alpha}^{\infty}$ and  $e_{-\alpha}^\infty$ must be  independent. The proof of lemma is finished.

\end{proof}

Let $\Gamma$ be the set of stabilizers of ${\hat x}_{0,\infty}$
\begin{align}\label{set-stability}
\Gamma=\{g\in G|~ {g\cdot {\hat x}_{0,\infty}}= {\hat x}_{0,\infty}\}.
\end{align}
Then it is a closed subgroup of $G$. Moreover, by Lemma \ref{nondegenerate-vector}, we have

\begin{cor}\label{local-open} Suppose that $\hat x_{0,\infty}\in \hat{\mathcal O}_\infty^0$.
Then there is a small neighborhood $U_{\rm Id}$ of ${\rm Id}$ in $G$ such that
$$\Gamma\cap U_{\rm Id}=\{{\rm Id}\}.$$
\end{cor}

\begin{proof}Let $\{E_1,...,E_n\}$ be a basis of $\mathfrak g$ chosen as in Lemma \ref{nondegenerate-vector}. Then there exists a small $\epsilon>0$ such that
$$(\exp\{z_1E_1\}......\exp\{z_nE_n\} )\hat x_{0,\infty}\subset \hat{\mathcal O}_\infty^0,~ \forall~ |z_l|\le \epsilon.$$
Since the set
$\{\exp\{z_1E_1\}......\exp\{z_nE_n\}||z_i|\le \epsilon\}\subset G$ covers an open set of ${\rm Id}$,
there exists $U_{\rm Id}\subset G$ such that
$$U_{\rm Id}\subset \{\exp\{z_1E_1\}......\exp\{z_nE_n\}||z_i|\le \epsilon\}.$$
Note that for any $g\in U_{\rm Id}$, there is a $|t_0|\le \epsilon$ and $X\in \mathfrak g$ such that $g=\exp\{t_0X\}$. By
(\ref{non-zero}),
$$ X(\hat x)\neq 0, ~\forall ~\hat x = \exp\{sX\}\cdot \hat x_{0,\infty}, ~|s|\le |t_0|.$$
It follows that
$$g\cdot \hat x_{0,\infty}\neq \hat x_{0,\infty},~\forall ~g\neq {\rm Id}\in U_{\rm Id}.$$ Thus $g\in\Gamma\cap U_{\rm Id}$ if and only if $g={\rm Id}$. The corollary is proved.

\end{proof}

By Corollary \ref{local-open}, $\Gamma$ is a discrete set. Next we show that
\begin{align}\label{finite-group}\#\Gamma =N_0<\infty.
\end{align}
We use the contradiction argument to prove (\ref{finite-group}) and suppose that $\#\Gamma=\infty$.
Then by Corollary \ref {local-open}, there is an infinite sequence of $\{g_l\in \Gamma | ~l\in\mathbb Z\}$ such that ${\rm dist}(g_l, {\rm Id})\to \infty$ as
$l\to \infty$. On the other hand, by the $KAK$ decomposition of $G$ \cite{Kna},  we see that there are $k_l,k_l'\in K$ and $a_l\in T$ such that
$g_l=k_l'\cdot a_l\cdot k_l$ and ${\rm dist}(a_l, {\rm Id})\to \infty$ as
$l\to \infty$. It follows that
there exists a $\delta_1>0$ such that
\begin{align}\label{distance} {\rm dist}(k_l\cdot \hat x_{0,\infty}, D \hat M_\infty)\ge \delta_1,~\forall~l.
\end{align}
In fact, if (\ref{distance}) is not true, there is a subsequence $\{k_{\alpha_l}\}$
which converges to $ k_0\in K$ and $k_0\cdot \hat x_{0,\infty}\in D \hat M_\infty.$
Note that any $g\in G$ fixes the set $D \hat M_\infty$.  Then
$$\hat x_{0,\infty}= k_0^{-1}( k_0 \cdot\hat x_{0,\infty})\in D \hat M_\infty,$$
which contradicts to the fact that $\hat x_{0,\infty}\in\hat{\mathcal O}_\infty$.

By (\ref{distance}), there is a compact set $\bar V\subset \hat{\mathcal O}_\infty$ such that $k_l\cdot \hat x_{0,\infty}\subset \bar V$
  for all $k_l$.  Furthermore, we have

\textbf{Claim 2:} For any small $\delta>0$, there is a large number $c_\delta$ such that
\begin{align}\label{c-closed}
{\rm dist}( a\cdot\hat y, D \hat M_\infty)\le \delta, ~\forall ~\hat y\in \bar V,
\end{align}
as long as $  {\rm dist}(a, {\rm Id})\ge c_\delta$,
where $a\in T$.

By (\ref {c-closed}), we see that there is a subsequence of integers  $\alpha_l$  such
that
$$a_{\alpha_l}(k_{\alpha_l}\cdot \hat x_{0,\infty}) \to \hat z\in D\hat M_\infty,~{\rm as}~\alpha_l \to \infty.$$
It follows that
$${\rm dist}( k_{\alpha_l}' [a_{\alpha_l}(k_{\alpha_l}\cdot \hat x_{0,\infty})], D\hat M_\infty)\to 0,
~{\rm as}~\alpha_l \to \infty.$$
But this is impossible since $ g_l\cdot \hat x_{0,\infty}=\hat x_{0.\infty}\in \hat{\mathcal O}_\infty.$ Hence, (\ref{finite-group})
is true.

To prove  \textbf{Claim 2},  we consider any element $ X\in \mathfrak g$ with ${\rm im}(X)\in \mathfrak k$ and its potential function  $f_X$ associated to the Fubini-Study metric $ \frac{1}{m} \omega_{FS}$ as in Proposition \ref{vector-limit} for the torus manifold  $\mathbb CP^N$.
Let  $M_X^{1}$ be a subset in $\mathbb CP^N$ defined by
$$M_X^1=\{x \in \mathbb CP^N | ~ f_X(x)= \max_{\mathbb CP^N} f_X\}.$$
Then
\begin{align}\label{zero-set-w}W_X=\{x\in \mathbb CP^N| ~x=\lim_{t\to\infty}\exp(t\text{re}(X))\cdot y, ~ {\rm for ~some}~y\in\hat{\mathcal O}_\infty^0 \}\subset M_X^{1}.
\end{align}
Moreover, if $X\in \mathfrak a$ is a torus vector field, it can be shown that $M_X^{1}$ is a subplane in $\mathbb CP^N$.

\begin{lem}\label{non-zero-vector} $W_X\cap\hat { \mathcal O}_\infty=\emptyset$ for any torus vector field $X\in \mathfrak a$.
\end{lem}
\begin{proof} On the contrary, we suppose that there is a point $\hat x\in W_{X} \cap\hat { \mathcal O}_\infty$. Then there is a point $\hat y_\infty
\in \hat{\mathcal O}_\infty^0$ such that
$$\hat x=\lim_{t\to\infty} \exp\{t{\rm re}(X)\}\cdot \hat y_\infty.$$
Let $X^i$ be a sequence of holomorphic vector fields on $\hat M^i$ which converges to $X$ with respect to  $K\times K$-invariant metric $(\Phi^{-1}_i)^*\omega_i$. Take a sequence of $\hat y_i\in \hat{\mathcal O}_i$ such that $\hat y_i\to \hat y_\infty$. Then by Proposition \ref{vector-limit}, there is a point $\hat x_i\in \hat M_{X^i}^1 $ for each $\hat y_i$ such that
$$\hat x_i=\lim_{t\to\infty} \exp\{t{\rm re}(X_i )\}\cdot \hat y_i,$$
where
\begin{align}\label{maximal-set-sequ} \hat M_{X^i}^1= \{x \in \hat M_i | ~ f_{X^i}(x)= \max_{\hat M_i} f_{X^i}\} \subset (\hat M_i \setminus \hat{\mathcal O}_i )=D\hat M_i,
\end{align}
and $f_{X^i}$ is a potential of $X^i$ with respect to the metric $(\Phi^{-1}_i)^*\omega_i$ on $\hat M_i$, which converges to a potential $f^\infty_{X}$ with respect to the metric $(\Phi^{-1}_\infty)^*\omega_\infty$ on $\hat M_\infty$.
Thus there is a limit $\hat x_\infty\in D\hat M_\infty\cap \hat M_{\infty,X}^1 $ of $\hat x_i$ in Gromov-Hausdorff topology,
 where $ (\hat M_{\infty, X}^1,(\Phi^{-1}_\infty)^*\omega_\infty)$
   is the Gromov-Hausdorff topology limit of $(\hat M_{X^i}^1, (\Phi^{-1}_i)^*\omega_i)$  as in Proposition \ref{vector-limit}.
   Note that
$$\max_{\mathbb CP^N} f_X= \max_{\hat M_\infty}f^\infty_{X}$$
since $(\Phi^{-1}_\infty)^*\omega_\infty$ and $\frac{1}{m} \omega_{FS}|_{\hat M_\infty}$ are both invariant under the $S^1$-group generated by ${\rm im}(X)$ \cite{Zhu00}.
Hence, we get
 \begin{align}\label{maximal-cpN}\hat M_{\infty,X}^1=\hat M_\infty\cap M_{X}^1.
 \end{align}
 Moreover,  by Proposition \ref{vector-limit} and (\ref{maximal-cpN}), we have
 \begin{align}\label{maximal-cpN-2}\max_{\hat M_i} f_{X^i}= \max_{\mathbb CP^N} f_X=A_0,~\forall x\in \hat M_{\infty,X}^1.
 \end{align}
It follows that $\hat x_\infty\in W_X $.

Choose a small neighborhood $T_\delta$ around the set $ \hat M_{\infty,X}^1$ in $\mathbb CP^N$ such that
\begin{align}\label{t-delta-ball} & 1)~f_X^\infty(x)\ge A_0- \delta, ~\forall x\in \hat M_\infty\cap T_\delta;\notag\\
& 2)~ f_X^\infty(x)= A_0- \delta, ~\forall x\in \hat M_\infty\cap \partial T_\delta;\notag\\
& 3)~ f_{X^i}(x)\ge A_0- 2\delta,~\forall x\in \hat M_i\cap T_\delta,~\forall i\ge i_0;\notag\\
& 4)~\hat x\not\in T_\delta.
\end{align}
3) can be guaranteed since $f_{X^i}$ converges to $f_{X}^\infty$ smoothly and $\hat M_{X^i}^1$ converges to $\hat M_{\infty,X}^1$ in Gromov-Hausdorff topology, and 4) is form the assumption that $\hat x\in W_{X} \cap\hat { \mathcal O}_\infty$.
Note that $f_{X^i}$ is monotone along the integral curve $\exp\{t{\rm re}(X^i )\}\cdot \hat y_i$. Then there is a uniform constant $T_N>0$ such that
$$\exp\{t{\rm re}(X^i )\}\cdot \hat y_i\subset T_\delta,~\forall ~t\ge T_N,~ i\ge i_0.$$
Thus we can choose a sequence of $z_i=\exp\{t_i{\rm re}(X^i )\}\cdot \hat y_i \in \hat M_i\cap (T_\delta\setminus T_{2\delta})$ which converges to a point $z_\infty\in \hat M_\infty\cap \overline{T_\delta}$, where $T_{2\delta}\subset T_{\delta}$ is another small neighborhood around $ \hat M_{\infty,X}^1$ in $\mathbb CP^N$ such that
$${\rm dist}( \hat M_{X^i}^1, \hat M_i\cap \partial T_{2\delta})\ge \delta'$$
for some sufficiently small $\delta'$. It follows that $t_i$ converges subsequently to some $T_0<\infty$ as $i\to\infty$.
On the other hand, by the convergence of $X^i$, we see that
$$\lim_i \exp\{t{\rm re}(X^i )\}\cdot \hat y_i= \exp\{t{\rm re}(X )\}\cdot \hat y_\infty,~\forall t\le 2T_0.$$
Hence we derive
$$z_\infty=\exp\{T_0{\rm re}(X )\}\cdot \hat y_\infty.$$
By the monotonicity of $f_{X}$ along the integral curve $\exp\{t{\rm re}(X )\}\cdot \hat y_\infty$, we conclude that
$$\exp\{t{\rm re}(X )\}\cdot \hat y_\infty\in \hat M_{\infty} \cap T_\delta, ~\forall t> T_0$$
and consequently, $\hat x\in \hat M_{\infty}\cap T_\delta.$ Therefore, we get a contradiction with 4) in (\ref{t-delta-ball}).
The lemma is proved.

\end{proof}

\begin{proof}[Proof of \textbf{Claim 2}] Suppose that \textbf{Claim 2} is not true. Then there exist a $\delta_0$, a sequence of $a_l\in T$ and a sequence of $\hat y_l\in \bar V$ such that
\begin{align}\label{not-19} {\rm dist}( a_l\cdot\hat y_l, D\hat M_\infty)\ge \delta_0,
\end{align}
where ${\rm dist} (a_l, \rm Id)\to\infty$ as $l\to\infty$. Write each $a_l$ as $a_l=\exp\{\sum b_l^i X_i\}$ for some real numbers $b^1_l,...,b^r_l$, where $\{X_1,...,X_r\}$ is a basis of $\mathfrak a$. Then
$\sum_i | b_l^i |\to\infty$ as $l\to\infty$. Without loss of generality, we may assume that
$$\sum b^i_l X_i=b^1_l (X_1+Y_l),$$
where $b^1_l\to \infty$ and $|Y_l|\to 0$ as $l\to\infty$. Then by Lemma \ref{non-zero-vector},
for any fixed $\hat y\in \bar V$ it holds
$${\rm dist}( \exp\{b_l^1(X_1+Y_l)\} \cdot\hat y, D\hat M_\infty \cap (\cup_{\alpha}\tilde W_\alpha))\to 0, ~{\rm as}~
b^1_l\to \infty,$$
where $ \tilde W_1,...,\tilde W_{N+1}$ are the $(N+1)$ hyperplanes in $\mathbb CP^N$ as in (\ref{zero-set-x}).
Since $\bar V$ is a compact set away from $\cup_{\alpha}\tilde W_\alpha$, as in the proof of {\textbf{Claim 1}} in Section 3, the above convergence is uniform. It follows that
$${\rm dist}( a_l\cdot\hat y_l, D\hat M_\infty)\to 0, ~ {\rm as}~
a_l\to \infty,$$
which contradicts to (\ref{not-19}).  {\textbf{Claim 2}} is proved.

\end{proof}

By (\ref{finite-group}), we can finish the proof of Proposition \ref{g-structure}.

\begin{proof}[Completion of proof of Proposition \ref{g-structure}]
For any $\hat x= h\cdot\hat x_{0,\infty}$, we have
$$ {(hgh^{-1})\cdot {\hat x} = (hgh^{-1}) ( h \cdot \hat x_{0,\infty} )}= {\hat x}, ~\forall ~g\in \Gamma.$$
It follows that $h\Gamma h^{-1}$ is the set of stabilizers of ${\hat x_\infty}$. By (\ref{finite-group}), $G\cdot \hat x_{0, \infty}$ is a finite quotient space. Since the above argument works for any $\hat x_\infty\in \hat{\mathcal O}_\infty^0$, in particular, both Corollary \ref{local-open} and (\ref{finite-group}) hold. Thus,
each orbit $G\cdot \hat x_\infty$ is isomorphic to $G/\Gamma_{\hat x_\infty}$,
where $\Gamma_{\hat x_\infty}$ is a finite subgroup of ${\rm PU}( N+1, \mathbb C)$. Moreover
$G\cdot \hat x_\infty\cap G\cdot \hat x_\infty'= \emptyset$ for any $\hat x_\infty, \hat x_\infty'\in \hat{\mathcal O}_\infty^0$. Otherwise
$G\cdot \hat x_\infty= G\cdot \hat x_\infty'.$ This means that any two different orbits are disjoint.
Note that
$$ \hat{\mathcal O}_\infty^0=\cup_{ \hat x_\infty\in\hat{\mathcal O}_\infty^0} G\cdot \hat x_\infty.$$
It is easy to see that for any bounded set $U$ in $\hat{\mathcal O}_\infty^0$ there are finitely many different orbits passing through $U$.
Since $\mathcal O_\infty\setminus \hat{\mathcal O}_\infty^0$ consists of  finitely many  subvarieties of codimension at least $1$ in $\mathcal O_\infty$, $\hat{\mathcal O}_\infty^0$ is connected. As a consequence, there is only one orbit $G\cdot \hat x_{0,\infty}$ through $U$. Otherwise $U$ will be disconnected. Therefore, we prove that $\hat{\mathcal O}_\infty^0= G\cdot \hat x_{0,\infty}$.

It remains to show that $\Gamma=\{{\rm Id}\}$ in (\ref{finite-group}). For any compact set $K_\infty^\epsilon\subset \hat{\mathcal O}_\infty^0$ with
\begin{eqnarray}\label{0112}
\text{vol}_{{\hat \omega}_\infty}(\hat M_\infty\backslash K_\infty^\epsilon)<\epsilon,
\end{eqnarray}
where $\hat \omega_\infty= \frac{1}{m} \omega_{FS}|_{ \hat M_\infty}$,  we choose a family of disjointed geodesic balls $\hat B_{r_l}$ in $ \hat {\mathcal O}_\infty^0$ such that the following holds:

1) $\sum_l {\rm vol}_{\hat\omega_\infty}(\hat B_{r_l}) \ge \text{vol}_{\hat \omega_\infty}( K_\infty^\epsilon) -\epsilon.$

2) For each $\hat B_{r_l}$, there are disjointed geodesic open sets $B_{r_l}^\alpha\subset \mathcal O$, $\alpha=1,..., N_0,$ such that $\pi^{-1}(\hat B_{r_l}) =\cup_\alpha B_{r_l}^\alpha$, where $\pi: ~ \mathcal O\to {\hat{\mathcal O}}_\infty^0$ is the projection by $ \lim_i\Phi_{i} (\Gamma\cdot x)=\Phi_\infty (\Gamma \cdot x_\infty)=\hat x_\infty$.

3) $(\Phi_i(B_{r_l}^\alpha), \hat{ \omega_i})$ is isometric to $(\Phi_i(B_{r_l}^\beta), \hat {\omega_i})$ for any $\alpha,\beta$.
\newline By Lemma \ref{nondegenerate-vector} (also see (\ref{converge-metric})), we see that
$(\Phi_i(B_{r_l}^\alpha), \hat \omega_i)$ converges to $(\hat B_{r_l}, \hat{\omega}_\infty)$ uniformly as open submanifolds when $i\to\infty$. In particular, it holds that for each $\alpha$,
$$\lim_i \sum_l {\rm vol}_{ \hat {\omega_i}} (\Phi_i(B_{r_l}^\alpha))=\sum_l {\rm vol}_{\hat{\omega}_\infty}(\hat B_{r_l}).$$
Note that $B_{r_l}^\alpha$ are disjointed for each $l, \alpha$.
Thus
\begin{align}\text{vol}_{\hat {\omega_i}}(\Phi_i(M)) &\ge \sum_{l,\alpha} {\rm vol}_{ \hat {\omega_i}} (\Phi_i(B_{r_i}^\alpha))\notag\\
&\ge N_0 \sum_l {\rm vol}_{\hat{\omega}_\infty}(\hat B_{r_l}) -N_0\epsilon\notag
\end{align}
as long as $i$ is large enough. By (\ref{0112}), it follows that
$$\text{vol}_{\hat {\omega_i}}(\Phi_i(M))\ge N_0 \text{vol}_{\hat{\omega}_\infty}(\hat M_\infty)- (2N_0+2)\epsilon.$$
But this is impossible if $N_0\ge 2$ since
\begin{align}\label{n0-1}\text{vol}_{\hat{\omega_i}}(\Phi_i(M))=\text{vol}_{\hat{\omega}_\infty}(\hat {M}_\infty)= c_1(M)^n.
\end{align}
Thus $\Gamma=\{\text{Id}\}$, and so $G$ acts on $\hat x_{0,\infty}$ freely. Hence, we prove that $(M_\infty,J_\infty)$ is a $G$-manifold.

\end{proof}

\section{Uniqueness of complex structures on    semisimple $G$-compactifications }
In this section, we first prove a uniqueness result about  complex structures on  $G$-manifolds when
 $G$ is semisimple. Then we complete the proof of  Theorem \ref{singular-type2}.

 We begin with following elemental lemma.

 \begin{lem}\label{uniqueness of splitting-structure}
Let $(Z,J)$ be an $r$-dimensional toric manifold with an $r$-dimensional torus $T^r$-action. Let $T^r=T^m\times T^{r-m}$ and $J'={\rm diag}(-J|_{T^m},J|_{T^{r-m}})$ be an integral almost complex structure on the open $T^r$-orbit $\mathcal O$ of $Z$. Suppose that $J'$ can be extended to a smooth complex structure on $Z$. Then $Z$ must be a product of $m$-dimensional toric manifold and $(r-m)$-dimensional toric manifold. Furthermore, $(Z,J)$ and $(Z,J')$ are bi-holomorphic.
\end{lem}

\begin{proof}
On the open $T^r$-orbit $\mathcal O$, we choose log-affine coordinates $w_1,...,w_r$. Let $\Sigma$ be the fan of $Z$ and $\sigma_a$ an $r$-dimensional cone in it. Then on the corresponding chart $U_a\subset Z$, we have local coordinates $z^1,...,z^r\in\mathbb C$ such that on $U_a\cap\mathcal O$,
\begin{eqnarray}\label{a-b-0}
\left\{\begin{aligned}
z^i&=\exp(\sum_jw^j \alpha_j^i+\sum_\beta w^\beta a_\beta^i)\\
z^\alpha&=\exp(\sum_jw^ja_j^\alpha+\sum_\beta w^\beta a_\beta^\alpha)
\end{aligned}\right.,
1\leq i,j\leq m<\alpha,\beta\leq r,
\end{eqnarray}
where
\begin{eqnarray}\label{a-b}
A=\left(\begin{aligned}&(a_i^j)_{m\times m}&(a_\alpha^j)_{m\times (r-m)}\\
&(a_i^\beta)_{(r-m)\times m}&(a_\alpha^\beta)_{(r-m)\times (r-m)}
\end{aligned}\right)\in GL_r(\mathbb Z).
\end{eqnarray}

On the open orbit $\mathcal O$, we have
\begin{eqnarray}\label{J'}
\begin{aligned}
J'=&\sqrt{-1}\left[-\sum_i(dw^i\otimes{\frac{\partial}{\partial w^i}}-d\bar w^i\otimes{\frac{\partial}{\partial \bar w^i}})\right.\\
&\left.+\sum_\alpha(dw^\alpha\otimes{\frac{\partial}{\partial w^\alpha}}-d\bar w^\alpha\otimes{\frac{\partial}{\partial \bar w^\alpha}})\right].
\end{aligned}
\end{eqnarray}
By \eqref{a-b}, it follows that
\begin{eqnarray}\label{J'_a}
\begin{aligned}
J'|_{U_a\cap\mathcal O}=&\sqrt{-1}\left[-\sum_i(dz^i\otimes{\frac{\partial}{\partial z^i}}-d\bar z^i\otimes{\frac{\partial}{\partial \bar z^i}})+\sum_\alpha(dz^\alpha\otimes{\frac{\partial}{\partial z^\alpha}}-d\bar z^\alpha\otimes{\frac{\partial}{\partial \bar z^\alpha}})\right]\\
&-4\text{Im}\left[\sum_{k,j}(A^{-1})^\alpha_k a^j_\alpha{\frac{z^j}{z^k}}dz^k\otimes{\frac{\partial}{\partial z^j}}+\sum_{i,\gamma}(A^{-1})^\alpha_ia^\gamma_\alpha{\frac{z^\gamma}{z^i}}dz^i\otimes{\frac{\partial}{\partial z^\gamma}}\right.\\
&-\left.\sum_{\beta,i}(A^{-1})^j_\beta a^i_j{\frac{z^i}{z^\beta}}dz^i\otimes{\frac{\partial}{\partial z^\beta}}+\sum_{\beta,\gamma}(A^{-1})^j_\beta a^\gamma_j{\frac{z^\gamma}{z^\beta}}dz^\beta\otimes{\frac{\partial}{\partial z^\gamma}}\right],
\end{aligned}
\end{eqnarray}
where $A^{-1}=((A^{-1})^p_q)$ is the inverse matrix of $A$ with elements $(A^{-1})^p_q$.
Note that $J'$ can be smoothly extended on whole $U_a$. By taking any variable $z^l$ of $\{z^1,...,z^r\}$ to $0$, it is easy to see that
\begin{align}\label{a-relation}
\left\{\begin{aligned}
&(A^{-1})_j^\alpha a_\alpha^k=0,~j\not=k\\
&(A^{-1})_j^\alpha a_\alpha^\gamma=0
\end{aligned}\right.
\text{ and }
\left\{\begin{aligned}
&(A^{-1})_\alpha^j a_j^k=0\\
&(A^{-1})^j_\alpha a_j^\beta=0,~\alpha\not=\beta
\end{aligned}\right..
\end{align}
Thus, by the fact $J'^2=-1$, we get
\begin{eqnarray}\label{J'-1}
\begin{aligned}
J'|_{U_a}=&\sqrt{-1}\left[\sum_i\epsilon_i(dz^i\otimes{\frac{\partial}{\partial z^i}}-d\bar z^i\otimes{\frac{\partial}{\partial \bar z^i}})\right.\\
&\left.+\sum_\alpha\epsilon_\alpha(dz^\alpha\otimes{\frac{\partial}{\partial z^\alpha}}-d\bar z^\alpha\otimes{\frac{\partial}{\partial \bar z^\alpha}})\right],
\end{aligned}
\end{eqnarray}
where each of $\epsilon_i$ and $\epsilon_\alpha$ is $1$ or $-1$.

By \eqref{J'} and (\ref{J'-1}), there must be $m$ numbers of $-1$ and $(r-m)$ numbers of $1$ in $\{\epsilon_1,...,\epsilon_r\}$. Without of loss of generality, we may assume that
$$\epsilon_i=-1,\epsilon_\alpha=1.$$
Then by \eqref{J'_a}, we get
\begin{eqnarray}\label{a'}
\left\{\begin{aligned}
&(A^{-1})_j^\alpha a_\alpha^k=0,\\
&(A^{-1})_j^\alpha a_\alpha^\gamma=0
\end{aligned}\right.
\text{ and }
\left\{\begin{aligned}
&(A^{-1})_\alpha^j a_j^k=0\\
&(A^{-1})^j_\alpha a_j^\beta=0,
\end{aligned}\right.~\forall ~ i,j,\alpha,\beta.
\end{eqnarray}
On the other hand, the matrices
\begin{eqnarray*}
\left(\begin{aligned}&( a_\alpha^j)_{m\times (r-m)}\notag\\
&(a_\alpha^\beta)_{(r-m)\times (r-m)}
\end{aligned}\right)
\text{ and }
\left(\begin{aligned}&( a_i^j)_{m\times m}\notag\\
&(a_i^\beta)_{(r-m)\times m}
\end{aligned}\right)
\end{eqnarray*}
are both of full ranks. Thus by \eqref{a'}, we have
$$(A^{-1})^\alpha_j=0,(A^{-1})_\alpha^j=0,$$
i.e.,
$$a^\alpha_j=0,a_\alpha^j=0.$$
As a consequence, by \eqref{a-b-0}, it follows that
\begin{eqnarray}\label{a-b-1}
\left\{\begin{aligned}
z^i&=\exp(\sum_jw^ja_j^i)\\
z^\alpha&=\exp(\sum_\beta w^\beta a_\beta^\alpha)
\end{aligned}\right.,~{\rm on}~U_a\cap \mathcal O.
\end{eqnarray}
Hence, the first equation in \eqref{a-b-1} defines a toric manifold $Z_1$ with $T^m$-action, while the second equation in \eqref{a-b-1} defines a toric manifold $Z_2$ with $T^{r-m}$-action. This proves that $Z=Z_1\times Z_2$.

By \eqref{a-b-1} the map
\begin{align}\label{tor-diff}
\Phi(z^i,z^\alpha)=(\bar z^i,z^\alpha)
\end{align}
is well-defined on $Z$, which satisfies $\Phi^*J'=J$.
Thus $(Z,J)$ and $(Z,J')$ are bi-holomorphic.

\end{proof}

 \begin{theo}\label{unique-complexstructure} Let $G$ be a semisimple reductive Lie group. Let $(M, K_M^{-1}, J)$ and  $(\tilde M,  K_{\tilde M}^{-1}, \tilde J)$ be two Fano compactifications of  $G$.  Suppose that $\tilde M$ is  diffeomorphic to $M$. Then  $(\tilde M, \tilde J)$ is  bi-holomorphic to $(M, J)$.

 \end{theo}

\begin{proof}Let $F: \tilde M\to M$ be a diffeomorphism. Then   it suffices to show that there is an  automorphism $\Psi$ on $M$ such that
\begin{eqnarray}\label{0203}
(F^{-1})^*{ \tilde J}=\Psi^*J, ~{\rm on}~ M.
\end{eqnarray}
Consider the $G\times G$-action on $M$, which is induced by the one on $ \tilde M$. Namely, for any $x\in \tilde M$, it holds
\begin{eqnarray}\label{0204}
{g\cdot F}(x)=F({g\cdot x}),~\forall ~g\in G\times G.
\end{eqnarray}
 Thus $J'=(F^{-1})^*\tilde J$ is also a $G\times G$-invariant integral almost complex structure on $M$ and it induces another complex structure on $G$.

Choose a base point $x_0\in\mathcal O$.
Since $G$ is a $2n$-dimensional real Lie group (denoted by $G_{\mathbb R}$) with an adjoint representation $\text{ad}_{\mathfrak g_{\mathbb R}}(\cdot)$ of $\mathfrak g_{\mathbb R}$ on itself, we have
$$J'(x_0)\in{\text{End}}'(\mathfrak g_{\mathbb R})=\{\sigma| ~\sigma\in {\text{End}}(\mathfrak g_{\mathbb R})\text{ and } \sigma(\text{ad}_XY)=\text{ad}_X(\sigma (Y)),~\forall ~X,Y\in{\mathfrak g_{\mathbb R}} \}.$$
On the other hand, the semisimple complex Lie algebra $\mathfrak g$ of $(G,J)$ can be decomposed into irreducible ideals $\mathfrak s_i$ of $\mathfrak g$, $i=1,...,l$, such that
$$\mathfrak g=\oplus_i\mathfrak s_i,$$
with $[\mathfrak s_i,\mathfrak s_j]=\delta_{ij}\mathfrak s_i$.
Then it is easy to see that
$${\text{End}}'(\mathfrak g_{\mathbb R})=\oplus_i{\text{End}}_i'(\mathfrak s_{i\mathbb R}), $$
where
$${\text{End}}_i'(\mathfrak s_{i\mathbb R})=\{\sigma| ~\sigma\in {\text{End}}(\mathfrak s_{i\mathbb R})\text{ and } \sigma(\text{ad}_XY)=\text{ad}_X(\sigma (Y)),~\forall ~X,Y\in{\mathfrak s_{i\mathbb R}} \}.$$
Note that each $\mathfrak s_i$ is a complex irreducible representation by $\text{ad}_{\mathfrak s_i}(\cdot)$ on $\mathfrak s_i$. Thus
$$\dim_{\mathbb C}{\text{End}}_i'(\mathfrak s_{i})=1. $$
As a consequence,
$\dim_{\mathbb R}{\text{End}}'(\mathfrak s_{i\mathbb R})=2$, which can be spanned by $\text{Id}$ and $J|_{\mathfrak s_{i\mathbb R}}$. Hence, $J'|_{\mathfrak s_{i\mathbb R}}=\lambda_i \text{Id}+\mu_i J|_{\mathfrak s_{i\mathbb R}}$ for some $\lambda_i,\mu_i\in\mathbb R$, and
$$J'|_{\mathfrak s_{i\mathbb R}}^2=(\lambda_i^2-\mu_i^2)\text{Id}+2\lambda_i\mu_i J|_{\mathfrak s_{i\mathbb R}}.$$
By the fact $J'^2=-\text{Id}$, it follows  that that $\lambda_i=0,\mu_i=\pm1$. Therefore,  we prove  that
\begin{eqnarray}\label{split}
J'(x_0)= \oplus \mu_i J|_{\mathfrak s_{i\mathbb R}}(x_0),
\end{eqnarray}
where $\mu_i=1, ~{\rm or}~-1.$

By \cite{AB1, AB2},   for the $G$-manifold $(M, J)$ there is an $r$-dimensional toric complex submanifold $(Z,J|_Z)$ through $x_0$ associated to a maximal torus $T^\mathbb C$ of $G$. Similarly,  there is another $r$-dimensional toric complex submanifold $(Z',J'|_{Z'})$ through $x_0$ associated to  a maximal torus  of the induced  $G$-action by (\ref{0204}).
 Moveover, by (\ref{split}), we have
\begin{align}
J'|_{TZ'} (x_0)= \oplus \mu_i J|_{\mathfrak t_{i\mathbb R}},\notag
\end{align}
where $\mathfrak t_{i\mathbb R}=\mathfrak s_{i\mathbb R}\cap\mathfrak t_{\mathbb R}$ which is non-empty for each $i$.
Thus there is a decomposition of $\mathfrak t_\mathbb R$ such that
\begin{eqnarray}\label{split-3}
\begin{aligned}
J'|_{TZ'} (x_0)&= (\oplus_{i=1}^{r_1} (-J)|_{\mathfrak t_{i\mathbb R}})\oplus(\oplus_{i=1}^{r_2} J|_{\mathfrak t_{i\mathbb R}})\\
&=(-J)|_{\mathfrak t^m_\mathbb R}\oplus J|_{\mathfrak t^{r-m}_\mathbb R},
\end{aligned}
\end{eqnarray}
where $\mathfrak t^m_\mathbb R$ and $\mathfrak t^{r-m}_\mathbb R$ are two Lie subalgebras of $\mathfrak t_\mathbb R$ with dimensions $m$ and $(r-m)$, respectively. Note that $Z|_{T^\mathbb C x_0}=Z'|_{T^\mathbb C x_0}$.  Hence,  by the completeness in the same ambient space $M$, we get
 \begin{align}\label{conjugate-Z}Z=\overline{Z|_{T^\mathbb C z_0}}=\overline{Z'|_{T^\mathbb C z_0}}=Z'\subset M.
 \end{align}
 By Lemma \ref{uniqueness of splitting-structure}, we prove that $(Z,J|_Z)$ and$(Z,J'|_{Z})$ are bi-holomorphic.

 Let  $P$  and $P'$  be two associated polytopes of  $(Z,K^{-1}_M|_Z,J|_Z)$ and   $(Z,K^{-1}_M|_{Z},J'|_{Z})$ as in Subsection 2.2, respectively.
  Since both of $m$-multiple bundles of $K^{-1}_M$ and $K^{-1}_{\tilde M}$ can be regarded as a restricted line bundle of $K^{-1}_{\mathbb CP^N}$ by the Kodaira embedding as in Section 4 for $M_i$ and $M_\infty$,
   $$(Z, K^{-m}_M|_Z, J|_Z)=(Z, K^{-1}_{\mathbb CP^N}|_Z),~
   (Z, K^{-m}_M|_{Z}, J'|_Z)=(F^{-1}(Z), K^{-1}_{\mathbb CP^N}|_{F^{-1}(Z)}).$$
   It follows that the curvatures of $(Z, K^{-1}_M|_Z, J|_Z)$ and $(Z, K^{-1}_M|_{Z}, J'|_Z)$ can be induced by  the Fubuni-Study of $\mathbb CP^N$ and so their  cohomology classes  on $Z$ are  same.  Thus  $P$ and $P'$ as the moment images of curvature forms  of the above line bundles are isomorphic.
Hence, by the equivariant classification theory \cite[Section 2]{AK}, two  polarized compactifications   $(M, K^{-1}_M,J)$ and $(M, K^{-1}_M,J')$   are different from  a $G\times G$-equivariant morphism.  Therefore,  $(M, J')$ must be bi-holomorphic to $(M, J)$.  Namely there is an  automorphism $\Psi$ on $M$ such that (\ref{0203}) holds. The theorem is proved.

\end{proof}

 Theorem \ref{unique-complexstructure} can be also proved  without using  the equivariant classification theory \cite[Section 2]{AK}. In fact, we can give a direct construction of  automorphism $\Psi$ by Lemma \ref{uniqueness of splitting-structure} and the Cartan involution in the following.

   Let $S_i$ be the subgroup of $G$ with Lie algebra $\mathfrak s_i$. Then each $S_i$ is semisimple and
\begin{align}\label{razlozhenie-G}
G=\prod_{i=1}^lS_i\slash {\rm diag}(\cap_iS_i),
\end{align}
where $\cap_iS_i$ is  a finite group (cf. \cite[Section 3.2]{Zhelobenko-Shtern}).
Fix a maximal compact subgroup $K_i$ in each $S_i$ and let  $\Theta_i$ be  the  Cartan involution on $S_i$  which acts trivially on $K_i$ and inverses  the complex structure of $S_i$.  We may choose $K_i$ so that $\cap_iK_i$ contains the finite group $\cap_i S_i$.  Note that    there are  $s_i\in S_i$ for any $g\in G$ such that $g=s_1\cdot...\cdot s_l$ by (\ref{razlozhenie-G}).  Thus we can define an automorphism on $G$ by
\begin{align}\label{razlozhenie-Theta}
\Theta(g)=\Theta^{\epsilon_1}_1(s_1)\cdot...\cdot\Theta^{\epsilon_l}_l(s_l),
\end{align}
where $\epsilon_i=1$ if $\mu_i=-1$,  and $\epsilon_i=0$ if $\mu_i=1$.  Since  each $s_i$ is uniquely determined up to multiplying an element of $\cap_iS_i$, which is fixed under $\Theta_i$, $\Theta$ is well-defined. Moreover, we have
 \begin{align}\label{conjugate-J}\Theta^*J=J'.
 \end{align}

It suffices to show that $\Theta$ can be extended to a diffeomorphism on  $M$. By Lemma \ref{uniqueness of splitting-structure}, we know  that $(Z,J|_Z)$ and $(Z,J'|_{Z})$ are bi-holomorphic.  Note that  the restriction of $\Theta$ on $T^\mathbb C$ is just $\Phi$ by\eqref{tor-diff}. Thus  $\Theta$ can be extended to a diffemorphism on $Z$ by
\begin{align}\label{extend-Theta}\Theta(z)=\Phi(z),~\forall z\in Z.
\end{align}
Moreover, $\Theta|_Z$ commutes with the $W$-action.

On the other hand, by a generalized KAK-decomposition of $G$-compactification (cf. \cite[Section 3.4]{Timashev-survey} or \cite[Section 9]{Timashev-Sbo}), for any $x\in M$, there are $k_1,k_2\in K$ and $z\in Z$ so that $x=(k_1,k_2)z$. Moreover, $z$ is uniquely determined up to a $W$-action. Since $\Theta$ commutes with the $W$-action and $\Theta|_K$ is trivial, the following $$\Theta(x)=(k_1,k_2)\Theta(z)$$
is well-defined by (\ref{extend-Theta}). Thus we can extend $\Theta$ to a diffeomorphism $\Psi$ on $M$ so that (\ref{0203}) holds by (\ref{conjugate-J}).  Hence, we also prove  Theorem \ref{unique-complexstructure}.

As a corollary   of Theorem \ref{unique-complexstructure}, we immediately get

\begin{prop}\label{bi-holo} The limit $(M_\infty, J_\infty)$ in Proposition \ref{g-structure} is bi-holomorphic to $(M, J)$ whenever $G$ is semisimple.
\end{prop}


By Proposition \ref{g-structure} and Proposition  \ref{bi-holo}, we are able to finish the proof of Theorem \ref{singular-type2}.

\begin{proof}[Proof of Theorem \ref{singular-type2}]
By the definition, if the solution of K\"ahler-Ricci flow (\ref{kahler-Ricci-flow}) has only type I singularities, then the curvature of $\omega(t)$
is uniformly bounded. Thus there is a subsequence $\{\omega(t_i)\}$ which converges to a
limit of K\"ahler-Ricci soliton $(M_\infty, \omega_\infty,J_\infty)$ in Cheeger-Gromov topology.
Note that the center of Lie algebra of the reductive part of ${\rm Aut}(M_\infty)$ is trivial by Proposition \ref{g-structure} since $G$ is semisimple. Thus $(M_\infty, \omega_\infty,J_\infty)$ must be a K\"ahler-Einstein metric. On the other hand, by Proposition \ref{bi-holo}, $(M_\infty, J_\infty)$ is biholomorphic to $(M,J)$, which admits no K\"ahler-Einstein metric by the assumption in the theorem. Hence, we get  a contradiction. As a consequence, the curvature of $\omega(t)$ must blow-up as $t\to \infty$. Namely, the solution of flow is of type II.

There is another way to prove Theorem \ref{singular-type2} without using Proposition  \ref{bi-holo}  if in addition we know that $M$ is $K$-unstable. In fact, the limit $(M_\infty, \omega_\infty,J_\infty)$ is a K\"ahler-Einstein metric if the solution of flow is of type I. Then by a result in \cite{CS} (also see \cite[Lemma 7.1]{TZ4}), the $K$-energy is bounded
below on the space of K\"ahler potentials in $2\pi c_1(M)$. This implies that $(M,J)$ is $K$-semistable \cite{DT, LX}. Thus we get a contradiction. Hence, the curvature of $\omega(t)$ must blow up as $t\to \infty$.

\end{proof}

\section{Examples of $G$-manifolds with rank 2}

In this section, we describe Fano compactifications of $\mathrm{SO}_4(\mathbb{C})$ and $\mathrm{Sp}_4(\mathbb{C})$.

\subsection{ Fano $\mathrm{SO}_4(\mathbb{C})$-manifolds of dimension 6}\label{exa1}

In \cite{De12}, Delcroix computed three polytopes $P_+$ associated to Fano compactifications of {$\mathrm{SO}_4(\mathbb{C})$ \footnote{In fact, by checking the Delzant condition of polytope $P$ and the Fano condition of compactified manifold, these three manifolds $M$ are only Fano compactifications of $\mathrm{SO}_4(\mathbb{C})$.}.} In the following, we write down the detailed data associated to $P_+$, in particular, the values of ${\rm bar}(P_+) $.

\begin{figure}[h]
\begin{center}
\begin{tikzpicture}
\draw [dotted] (0,-3) grid[xstep=1,ystep=1] (3,3);
\draw (0,0) node{$\bullet$};
\draw (2,0) node{$\bullet$};
\draw (1.6,0.3) node{$2\rho$};
\draw [semithick] (3,3) -- (0,0) -- (3,-3) -- (3,3);
\draw (2.7,2.3) node{$P_+$};
\draw (0.7,2.7) node{(1)};
\draw [very thick, -latex] (0,0) -- (1,-1);
\draw [very thick, -latex] (0,0) -- (1,1);
\end{tikzpicture}
\end{center}
\caption{}
\label{BarX1-2-4}
\end{figure}

Choose a coordinate on $\mathfrak a^*$ such that the basis are the generator of $\mathfrak M$. Then the positive roots are
\begin{eqnarray*}
\alpha_1=(1,-1),~\alpha_2=(1,1),
\end{eqnarray*}
and
\begin{eqnarray*}
2\rho=(2,0).
\end{eqnarray*}
Thus
\begin{eqnarray*}
&\mathfrak a_+^*=\{x>y>-x\},\\
&2\rho+\Xi=\{-2+x>y>2-x\},
\end{eqnarray*}
and
\begin{eqnarray*}
\pi(x,y)=(x-y)^2(x+y)^2.
\end{eqnarray*}

(A)-\textbf{Case (1).} There is one smooth Fano compactification of
$\mathrm{SO}_4(\mathbb{C})$, which admits a K\"ahler-Einstein metric.
The polytope $P_+$ is given by (See Figure~\ref{BarX1-2-4}),
\begin{align}\label{caseso1}
P_+=\{y>-x,x> y, 2-x>0,2+y>0\}.
\end{align}
A direct computation shows that  ${\rm vol}(P_+)=\frac{648}5$ and
\begin{eqnarray*}
{\rm bar}(P_+)=\left(\frac{18}7,0\right).
\end{eqnarray*}
Then
$$\label{bar}
{\rm bar}(P_+)\in 2\rho+\Xi$$
which implies (\ref{bar-1}). Thus by Theorem \ref{de}, the $\mathrm{SO}_4(\mathbb{C})$-manifold associated to $P_+$ in (\ref{caseso1}) admits a K\"ahler-Einstein metric.

(B) There are two smooth Fano compactifications of $\mathrm{SO}_4(\mathbb{C})$ with no K\"ahler-Einstein metrics.
Both of $P_+$ (see Figure~\ref{BarX1-2}) do not satisfy (\ref{bar-1}). Moreover, The
Futaki invariant vanishes since the center of
automorphisms group are finite. Hence there are also no K\"ahler-Ricci solitons on the compactifications.

\begin{figure}[h]
\begin{center}
\begin{tikzpicture}
\draw [dotted] (0,-2) grid[xstep=1,ystep=1] (3,3);
\draw (0,0) node{$\bullet$};
\draw (2,0) node{$\bullet$};
\draw (1.6,0.3) node{$2\rho$};
\draw [semithick] (3,3) -- (3,0) -- (3/2,-3/2) -- (0,0) -- (3,3);
\draw (2.7,2.3) node{$P_+$};
\draw (0.7,2.7) node{(2)};
\draw [very thick, -latex] (0,0) -- (1,-1);
\draw [very thick, -latex] (0,0) -- (1,1);
\end{tikzpicture}
\begin{tikzpicture}
\draw [dotted] (0,-2) grid[xstep=1,ystep=1] (3,3);
\draw (0,0) node{$\bullet$};
\draw (2,0) node{$\bullet$};
\draw (1.6,0.3) node{$2\rho$};
\draw [semithick] (3,3) -- (3,1) -- (2,-1) -- (3/2,-3/2) --(0,0) -- (3,3);
\draw (2.7,2.3) node{$P_+$};
\draw (0.7,2.7) node{(3)};
\draw [very thick, -latex] (0,0) -- (1,-1);
\draw [very thick, -latex] (0,0) -- (1,1);
\end{tikzpicture}
\end{center}
\caption{}
\label{BarX1-2}
\end{figure}

\textbf{Case (2).} The polytope is
\begin{eqnarray*}
P_+=\{y>-x,x> y, 2-x>0,2+y>0,3-x+y>0\}.
\end{eqnarray*}
Then   ${\rm vol}(P_+)=  \frac{1701}{20}$ and  the barycenter is
\begin{eqnarray*}
{\rm bar}(P_+)=\left(\frac{489}{196},\frac{15}{28}\right).
\end{eqnarray*}
Thus
$$\label{bar}
{\rm bar}(P_+)\not\in \overline{2\rho+\Xi}$$
and consequently, there is no K\"ahler-Einstein metric in \textbf{Case (2)}.

\textbf{Case (3).} The polytope is
\begin{eqnarray*}
P_+=\{y>-x,x> y, 2-x>0,2+y>0,3-x+y>0,5-2x+y>0\}.
\end{eqnarray*}
Then    ${\rm vol}(P_+)=   \frac{10751}{180}$ and   the barycenter is
\begin{eqnarray*}
{\rm bar}(P_+)=\left(\frac{102741}{43004},\frac{16575}{23156}\right).
\end{eqnarray*}
Thus
$$\label{bar}
{\rm bar}(P_+)\not\in \overline{2\rho+\Xi}$$
and consequently, there is no K\"ahler-Einstein metric in \textbf{Case (3)}.

\subsection{ Fano $\mathrm{Sp}_4(\mathbb{C})$-manifolds of dimension 10}\label{exa2}
In \cite{Del1} Delcroix computed two polytopes $P_+$ associated to toroidal Fano compactifications of $\mathrm{Sp}_4(\mathbb{C})$ (see Cases (1) and (3) below). By the tables listed in \cite{Ruzzi}, one can check that there are in total three smooth Fano compactifiactions.
We give the data in the following. The positive roots are
\begin{eqnarray*}
\alpha_1=(1,-1),\alpha_2=(2,0),\alpha_3=(1,1),\alpha_4=(4,2).
\end{eqnarray*}
Consequently, $\rho=(2,1)$,
$$2\rho+\Xi=\{y>6-x,x>4\}$$
and $\pi(x,y)=16(x-y)^2(x+y)^2x^2y^2$.

(A) There are two smooth Fano compactification which admit K\"ahler-Einstein metrics (see Figure \ref{BarX3}).
\begin{figure}[h]
\begin{center}
\begin{tikzpicture}
\draw [dotted] (0,-1) grid[xstep=1,ystep=1] (5,5);
\draw (0,0) node{$\bullet$};
\draw (4,2) node{$\bullet$};
\draw (4.2,2.2) node{$2\rho$};
\draw [semithick] (0,0) -- (5,0) -- (5,2) -- (7/2,7/2) -- (0,0);
\draw (4.7,3.3) node{$P_+$};
\draw (0.7,2.7) node{(1)};
\draw [very thick, -latex] (0,0) -- (1,-1);
\draw [very thick, -latex] (0,0) -- (0,2);
\draw [very thick, -latex] (0,0) -- (1,1);
\draw [very thick, -latex] (0,0) -- (2,0);
\draw [very thick, -latex] (0,0) -- (4,2);
\end{tikzpicture}
\begin{tikzpicture}
\draw [dotted] (0,-1) grid[xstep=1,ystep=1] (5,5);
\draw (0,0) node{$\bullet$};
\draw (4,2) node{$\bullet$};
\draw (4.2,2.2) node{$2\rho$};
\draw [semithick] (0,0) -- (5,0) -- (5,5) -- (0,0);
\draw (4.7,3.3) node{$P_+$};
\draw (0.7,2.7) node{(2)};
\draw [very thick, -latex] (0,0) -- (1,-1);
\draw [very thick, -latex] (0,0) -- (0,2);
\draw [very thick, -latex] (0,0) -- (1,1);
\draw [very thick, -latex] (0,0) -- (2,0);
\draw [very thick, -latex] (0,0) -- (4,2);
\end{tikzpicture}
\end{center}
\caption{}
\label{BarX3}
\end{figure}

\textbf{Case (1).} The polytope $P_+$ is given by
\begin{align*}
P_+=\{y>0,x> y, 5-x>0,7>x+y\}.
\end{align*}
A direct computation shows that  ${\rm vol}(P_+)=   \frac{31702283}{1400}$ and
\begin{eqnarray*}
{\rm bar}(P_+)=\left(\frac{456413622265}{104829824704},\frac{186115662215}{104829824704}\right)\in 2\rho+\Xi,
\end{eqnarray*}
which implies (\ref{bar-1}). Thus by Theorem \ref{de}, the $\mathrm{Sp}_4(\mathbb{C})$-manifold associated to $P_+$ in \textbf{Case (1)} admits a K\"ahler-Einstein metric.

\textbf{Case (2).} The polytope $P_+$ is given by
\begin{align*}
P_+=\{y>0,x> y, 5>x\}.
\end{align*}
A direct computation shows that   ${\rm vol}(P_+)=   \frac{1562500}{21}$ and
\begin{eqnarray*}
{\rm bar}(P_+)=\left(\frac{50}{11},\frac{875}{352}\right)\in 2\rho+\Xi.
\end{eqnarray*}
Hence there admits a K\"ahler-Einstein metric in \textbf{Case (2)}.

(B)-\textbf{Case (3).} There is one smooth Fano compactification which does not admit K\"ahler-Einstein metrics (see Figure \ref{BarX4}).
\begin{figure}[h]
\begin{center}
\begin{tikzpicture}
\draw [dotted] (0,-1) grid[xstep=1,ystep=1] (5,5);
\draw (0,0) node{$\bullet$};
\draw (4,2) node{$\bullet$};
\draw (4.2,2.2) node{$2\rho$};
\draw [semithick] (0,0) -- (5,0) -- (5,1) -- (4,3) -- (7/2,7/2) -- (0,0);
\draw (4.7,3.3) node{$P_+$};
\draw (0.7,2.7) node{(3)};
\draw [very thick, -latex] (0,0) -- (1,-1);
\draw [very thick, -latex] (0,0) -- (0,2);
\draw [very thick, -latex] (0,0) -- (1,1);
\draw [very thick, -latex] (0,0) -- (2,0);
\draw [very thick, -latex] (0,0) -- (4,2);
\end{tikzpicture}
\end{center}
\caption{}
\label{BarX4}
\end{figure}

The polytope $P_+$ is given by
\begin{align*}
P_+=\{y>0,x> y, 5-x>0,7>x+y,11>2x+y\}.
\end{align*}
A direct computation shows that  ${\rm vol}(P_+)=   \frac{148906001}{4200}$ and
\begin{eqnarray*}
{\rm bar}(P_+)=\left(\frac{278037566905}{66955221696},\frac{111498923355}{66955221696}\right)\not\in \overline{2\rho+\Xi}.
\end{eqnarray*}
Hence there does not admit a K\"ahler-Einstein metric in \textbf{Case (3)}.

Two $\mathrm{SO}_4(\mathbb{C})$-manifolds in Section \ref{exa1} (B-Cases (2), (3)) and one $\mathrm{Sp}_4(\mathbb{C})$-manifold in Section \ref{exa2} (B-Case (3)) are those examples described as in
Theorem \ref{SO(4)}. Moreover, these three Fano manifolds are all $K$-unstable.

\subsection{Remarks on Theorem \ref{singular-type2} and  Theorem \ref{SO(4)}}
 By Theorem  \ref{singular-type2} and the Hamilton-Tian conjecture \cite{T1, TZhzh, Bam, Chwang, WZ20}, the  K\"ahler-Ricci flow (\ref{kahler-Ricci-flow})
  will converge to  a $Q$-Fano variety $\tilde M_\infty$  with a singular  K\"ahler-Ricci soliton.  One may expect  that  $\tilde M_\infty$ is a
$\mathbb Q$-Fano compactification of $G$ by extending the argument in the proof of Proposition \ref{g-structure}.  Unfortunately, it is not true in general. In fact, in a sequel of  paper \cite{LTZ2},  we prove

\begin{theo}\label{LTZ3}
There is no $\mathbb Q$-Fano compactification of $SO_4(\mathbb C)$ which admits a singular K\"ahler-Einstein metric with  the same volume  as in
Section \ref{exa1} (B-Cases (2), (3)).
\end{theo}

By the Hamilton-Tian conjecture,  the limit $M_\infty$  of (\ref{kahler-Ricci-flow}) will   preserve the
volume (also see \cite[Theorem 1.1]{ WZ20}).  Thus if $M_\infty$  is a
$\mathbb Q$-Fano compactification of $SO_4(\mathbb C)$ in case of  $G=SO_4(\mathbb C)$, there will be a contradiction with  Theorem \ref{LTZ3}.
 Theorem \ref{LTZ3} implies that the limit soliton
will has less symmetry than the original one, which is totally different to the situation of smooth convergence as in Proposition   \ref{g-structure}.
However, it is still interesting in understanding the $\mathbb Q$-Fano structure of   $M_\infty$.

 Although  we shall assume  that metrics  $(M,\omega_i, J)$ are all $K\times K$-invariant   in the proofs of both of  Proposition  \ref{g-structure} and Proposition \ref{bi-holo}, the  $K\times K$-invariant condition for the
 initial metric $\omega_0$ in Theorem \ref{singular-type2} can be removed   by using a recent result for  the uniqueness of limits of K\"ahler-Ricci flow  with varied initial metrics in \cite{WangZ2, HL}.  In fact, we have

  \begin{theo}\label{singular-type2-2} Let $G$ be a complex semisimple Lie group and
$M$ a Fano $G$-manifold which admits no K\"ahler-Einstein metrics.  Then any solution of K\"ahler-Ricci flow (\ref{kahler-Ricci-flow}) on $M$ with an  initial metric $\omega_0\in 2\pi c_1(M)$ is of type II.

\end{theo}

 \begin{proof} By Theorem \ref{singular-type2}, we claim that the $Q$-Fano variety limit   $\tilde M_\infty$  of   flow $(M, \omega(t))$  with a $K\times K$-invariant initial metric  in  the Hamilton-Tian conjecture  is a singular variety.  In fact,  on contrary, by the partial $C^0$-estimate in \cite{WangZ2},
 the Kodaira images $\tilde M_t$  in $\mathbb CP^N$ associated to $\omega(t)$ will smoothly converge to  $\tilde M_\infty$. Then  we get the estimates
 (\ref{induced-metrics}) and (\ref{higher-regularity})
 for metrics $\omega(t)$ as in Lemma \ref{metric-equiv}. In particular, the curvature of $\omega(t)$ is uniformly bounded, which is contradict with Theorem \ref{singular-type2}!

 On the other hand,  by \cite{WangZ2, HL},  the singular $Q$-Fano variety limit   $\tilde M_\infty$  of   (\ref{kahler-Ricci-flow})
  is independent of the choice of initial metric $\omega_0$.  Then by a result \cite[Lemma 6.2]{WangZ2}, the Gromov-Hausdroff limit $(M_\infty, \omega_\infty)$ of any sequence of  flow $(M, \omega(t))$ with any initial metric $\omega_0$  could not be a smooth Riemannian manifold since  $\tilde M_\infty$  is a singular variety. This implies that $(M, \omega(t))$ is of type II.

 \end{proof}

\section{Appendix:  An analytic  proof of Proposition \ref{bi-holo}  by Gang Tian and Xiaohua Zhu}

Fix a point $\hat x_\infty\in \hat{\mathcal O}_\infty^0$ as in Lemma \ref{nondegenerate-vector} and choose a sequence of $x_i\in\mathcal O\subset M$ such that $\hat x_i=\Phi_i(x_i)\in \hat M_i\to \hat x_\infty$. Then by the relation (\ref{a0}), $\psi^i$ converges to a convex function $\psi^\infty$ on $D_\epsilon\subset {\mathfrak a}$ with the property
\begin{align}\label{gradient-limit}\nabla\psi^\infty(0)=\lim_i\nabla\psi^i(x_i),
\end{align}
where  $D_\epsilon$ is a small $\epsilon$ ball centered  at the original with coordinates $(y^1,...,y^r)$  in ${\mathfrak a}$. By the regularity in Lemma \ref{metric-equiv},  $\psi^\infty$ is smooth and it satisfies that
$$\nabla^2\psi^\infty(\frac{\partial}{\partial y^a}, \frac{\partial}{\partial y^b})=\omega_\infty(e_a^\infty, e_b^\infty)= \lim_i \psi_{ab}^i.$$
On the other hand, by the relation (\ref{action def}), it is easy to see that the limit metric $\omega_\infty$ is also  $K\times K$-invariant. Thus by the uniqueness of $K\times K$-invariant functions associated to  $\omega_\infty$,  $\psi^\infty$ (modulo a constant)  can be uniquely  extended to a Weyl-invariant convex function  on ${\mathfrak a}$ such that its associated $K\times K$-invariant function $\Psi_\infty$ on $G$ satisfies
$$\omega_\infty=\sqrt{-1}\partial\bar\partial\Psi_\infty, ~{\rm on}~ G. $$

Recall the associated polytope $P\subset {\mathfrak a}^*$ of $r$-dimensional  torus  complex submanifold $Z$ in $M$ in Section 2.2. Then
$${\rm Im}(\nabla\psi^i)=2P, $$
 which is independent of $i$. We shall prove

\begin{lem}\label{invariant-polytope}
$${\rm Im}(\nabla\psi^\infty)=2P.$$
\end{lem}

\begin{proof}1) ${\rm Im}(\nabla\psi^\infty)\subset 2P.$ This is clear by (\ref{gradient-limit}).  In fact,  for any $\hat x_\infty\in \hat{\mathcal O}_\infty^0$ there is a sequence of $ x_i\in M \to x_\infty=\Phi_\infty^{-1}(\hat x_\infty)$ such that
\begin{align}\label{one-side}\nabla\psi^\infty( x_\infty)=\lim_i\nabla\psi^i(x_i)\in 2P.
\end{align}

2) $ 2P \subset {\rm Im}(\nabla\psi^\infty).$  On contrary, we suppose that there is a point $v\in 2P\setminus \overline{{\rm Im}(\nabla\psi^\infty)}$. Then   we can choose a square set $\Delta$ around $v$ by
$$\Delta=\{v'\in 2P|~|(v')^a-v^a|\le \delta, a=1,...,r\}\subset 2P\setminus \overline{{\rm Im}(\nabla\psi^\infty)}.$$
Let
$$U_i=\{x\in\mathcal O |~ \nabla \psi^i(x)\in \Delta\}.$$
By  Lemma \ref{Hessian},
we get
\begin{align}\label{volume}
{\rm  vol}_{\omega_i}(U_i)=&C_0\int_{B_i} \prod_{\alpha\in\Phi_+}\langle  \alpha, \nabla \psi^i\rangle^2  {\rm det}(\nabla^2\psi^i) dy\notag\\
&=C_0 \int_{\Delta}\pi(y')dy'\ge \delta_0,
\end{align}
where $C_0$ and $\delta_0$ are constants, and $B_i=\{y\in {\mathfrak a}|~\nabla \psi^i(y)\in \Delta\}$.

 We claim that  there are  $\epsilon_0>0$  and  a sequence of $x_i\in U_i$ such that
\begin{align}\label{gap-distance}{\rm dist}(x_i,  M\setminus {\mathcal O})\ge \epsilon_0.
\end{align}
In fact, if (\ref{gap-distance}) is not true, then there is a subsequence of sets $U_{i}$ (still denoted by the same sequence of $U_{i}$  for convenience)  such that
$$ {\rm dist}(x',  M\setminus {\mathcal O})\to 0, ~\forall~ x'\in U_{i}.$$
Then we can choose a sequence of  $\epsilon_{i}$-tubular neighborhood $T_i$ of  $M\setminus {\mathcal O}$ with $\epsilon_i\to 0$ such that
$$U_{i}\subset T_i.$$
Since $(T_i, \omega_i)$ converges to $(DM_\infty, \omega_\infty)$ in the Gromov-Hausdroff topology, where $DM_\infty=\Phi_\infty^{-1}$
\newline $(D\hat M_\infty )$,  ${\rm vol}_{\omega_i}( T_i)$ goes  to zero by the volume convergence theorem of  Colding \cite{Co97}. But this is impossible by  (\ref{volume}).

By (\ref{gap-distance}), we see that
$${\rm dist}_{\omega_i}(B_{\frac{\epsilon_0}{2}}(x_i),  M\setminus {\mathcal O})\ge \frac{\epsilon_0}{4}.$$
Let $ x_\infty$ be the limit of $x_i$. Then $\hat x_\infty\not\in D\hat M_\infty$.  Since  we already know that $ M_\infty$ is a G-manifold, $\hat x_\infty\in  \hat {\mathcal O}_\infty^0$. Now we can use (\ref{one-side}) in the above 1) to  conclude that
$$\lim_i\nabla\psi^i(x_i)\in {\rm Im}(\nabla\psi^\infty).$$
However, this is impossible since  each $\nabla\psi^i(x_i)\not\in \overline{{\rm Im}(\nabla\psi^\infty)}.$

\end{proof}

By the equivariant classification theory \cite[Section 2]{AK}, the  polarized $G$-compactifications   $(M, $
\newline $K^{-1}_M,J)$  is determined by the associated polytope $P$ of $(Z,K^{-1}_{M}|_{Z},J|_{Z})$ as in Subsection 2.2.  Let  $(Z',J_\infty|_{Z'})$ be  an $r$-dimensional toric complex submanifold of $M_\infty$ generated by torus vector fields through a point  $x_\infty\in {\mathcal O}_\infty^0 \subset M_\infty$.
 Then by Lemma \ref{invariant-polytope}, the associated polytope of  $(Z',K^{-1}_{M_\infty}|_{Z'},J_\infty|_{Z'})$ is same as $P$.
Thus  we  prove

\begin{prop}\label{bi-holo-new} The limit $(M_\infty, J_\infty)$ in Proposition \ref{g-structure} is bi-holomorphic to $(M, J)$.
\end{prop}

Proposition \ref{bi-holo-new} removes the assumption that $G$ is semisimple  in Proposition  \ref{bi-holo}, so  we can improve Theorem \ref{singular-type2} (also see  Theorem \ref{singular-type2-2}) as follows.

\begin{theo}\label{singular-type2-new} Let $G$ be a complex reductive  Lie group and
$M$  a Fano $G$-manifold which admits no K\"ahle-Ricci soliton.  Then any solution of K\"ahler-Ricci flow (\ref{kahler-Ricci-flow}) on $M$ with any
 initial metric $\omega_0\in 2\pi c_1(M)$ is of type II.
\end{theo}

\end{document}